\documentclass[a4paper,10pt]{amsart}

\usepackage{amssymb,amsmath,amsthm,mathrsfs,enumerate,multicol,esint}

\allowdisplaybreaks

\title[$T1$ criterions for Hardy and BMO spaces]{$T1$ criterions for generalised Calder\'on--Zygmund type operators on Hardy and BMO spaces associated to Schr\"odinger operators and applications}

\author{The Anh Bui}
\address{Department of Mathematics, 
Macquarie University, 
NSW 2109, Australia}
\email{the.bui@mq.edu.au}

\author{Ji Li}
\address{Department of Mathematics, Macquarie University, 
NSW 2109, Australia}
\email{ji.li@mq.edu.au}

\author{Fu Ken Ly}
\address{Sydney Institute of Business and Technology,   
Level 2, 11 York Street, 
NSW 2000, Australia}
\email{ken.ly@learning.sibt.nsw.edu.au}


\newcommand{\NormA}[1]{\lVert #1 \rVert}
\newcommand{\AbbsA}[1]{\lvert #1 \rvert}
\newcommand{\ContainC}[1]{\Bigl( #1 \Bigr)}

\newcommand{\RR}[1]{\mathbb{R}^{#1}} %Euclidean space
\newcommand{\MM}{\mathcal{M}}
\newcommand{\RH}{RH} % Reverse Holder class macro
\newcommand{\CMS}[1]{BMO^{#1}_L}
\newcommand{\Lip}[1]{\Lambda^{#1}}
\newcommand{\vc}{\infty}
\newcommand{\f}{\frac}
\newcommand{\lesi}{\lesssim}
\newcommand{\RRT}{\mathcal{R}_{(2)}}
\newcommand{\RT}{\mathcal{R}_{(1)}}
\DeclareMathOperator{\supp}{supp\,} %support of a function

\theoremstyle{plain}
\newtheorem{Theorem}{Theorem}[section]
\newtheorem{Lemma}[Theorem]{Lemma}
\newtheorem{Proposition}[Theorem]{Proposition}

\newtheorem{Definition}[Theorem]{Definition}
\theoremstyle{definition}
\newtheorem{Remark}[Theorem]{Remark}
\theoremstyle{remark}

\numberwithin{equation}{section}	%allows numbering of equations based on section

\def\barint{\kern4pt
\raise3.4pt\hbox{\vrule height.8pt width5pt}%
\kern-9pt % -(4pt + 5pt)
\int}

\begin{document}

\begin{abstract}
 Suppose $L=-\Delta+V$ is a Schr\"odinger operator on $\mathbb{R}^n$ with a potential $V$ belonging to certain reverse H\"older class $RH_\sigma$ with $\sigma\geq n/2$. The main aim of this paper is to provide necessary and sufficient conditions in terms of $T1$ criteria for a generalised Calder\'on--Zygmund type operator with respect to $L$  to be bounded on Hardy spaces $H^p_L(\mathbb{R}^n)$ and on BMO type spaces BMO$_L^\alpha(\mathbb{R}^n)$ associated with $L$.  As applications, we prove the boundedness for several singular integral operators associated to $L$. Our approach is flexible enough to prove the boundedness of the Riesz transforms related to $L$  with $n/2 \leq \sigma <n$ which were investigated in \cite{MSTZ} under the stronger condition $\sigma\geq n$. Thus our results not only recover existing results in \cite{MSTZ} but also contains new results in literature.
\end{abstract}

\maketitle
\tableofcontents

\section{Introduction}

It is well-known that the $T1$ theorem plays a crucial role in the analysis of $L^2$ boundedness (and furthermore the $L^p$ boundedness) of Calderon--Zygmund singular integral operators (see \cite{DJ} and \cite[p. 590]{G}).  For the endpoint boundedness (i.e. $p=1$ and $p=\infty$), there are also analogous $T1$ criterions for Calder\'on--Zygmund operators. To be more precise, suppose $T$ is a Calder\'on--Zygmund operator (in the sequel we denote this by $T\in CZO$), then $T$ is bounded on the Hardy space $H^1(\mathbb{R}^n)$ if and only if $T^*1=0$, and bounded on the BMO space BMO$(\mathbb{R}^n)$ if and only if $T1=0$ (see for example \cite{HHL}). 

Recently, Betancor et al. \cite{BCFST} established a $T1$ criterion for Hermite--Calder\'on--Zygmund operators on the BMO space BMO$_H(\mathbb{R}^n)$ associated to the Hermite operator (also known as harmonic oscillator) $H= -\Delta + |x|^2$ in $\mathbb{R}^n$. Based on this criterion they studied systematically the boundedness of certain singular integral operators related to $H$ on BMO$_H(\mathbb{R}^n)$, such as
Riesz transforms, maximal operators related to the heat and Poisson semigroups, Littlewood--Paley $g$-functions, as well as variation operators. This $T1$
 criterion was generalised by Ma et al.  \cite{MSTZ}, where they established a $T1$ criterion for boundedness in the Campanato type spaces BMO$_L^\alpha(\mathbb{R}^n)$ of so-called $\gamma$-Schr\"odinger--Calder\'on--Zygmund operators, which are related to  
 the Schr\"odinger operator $L$ on $\mathbb{R}^n$, $n\geq3$, given by
\begin{equation}\label{L}
 L=-\Delta+V, \qquad V\in\RH_\sigma, \ \ \sigma \ge n/2. 
\end{equation}
 The expression $V \in RH_\sigma$ means that $V$ is a non-negative function that satisfies the reverse H\"older inequality 
 \begin{align}\label{RH} \bigg( {1\over |B|} \int_B V(y)^\sigma dy \bigg)^{1\over \sigma} \leq {C\over |B|} \int_B V(y)dy.   
\end{align} 
for some constant $C=C(q,V)$ and every ball $B$. 

 As applications, they obtained  regularity estimates for certain operators related to $L$ such as the maximal operators and square functions of the heat and Poisson semigroups, for Laplace transform type multipliers,  for negative powers $L^{-\gamma/2}$. Moreover, on restricting  $\sigma\ge n$, they obtained regularity estimates for the Riesz transforms $\nabla L^{-1/2}$.

Shen \cite{Sh} proved that when $\sigma\ge n$, the Riesz transforms $\nabla L^{-1/2}$ are Calder\'on--Zygmund operators. However, this may not be true when $n/2\leq\sigma <n$ because pointwise estimates on the kernel of $\nabla L^{-1/2}$ are not available. But certain weaker estimates related to the standard H\"ormander condition 
\begin{align}\label{Hormander}
\int_{|x-y|>\delta |y-\overline{y}|}|K(x,y)-K(x,\overline{y})|dx\leq C
\end{align}
have been derived in \cite{BHS,GLP}, for some $C>0$ and $\delta>1$ and every $y,\overline{y}\in \mathbb{R}^n$, . 

The aim of this article is to provide necessary and sufficient conditions for a larger class of generalised Calder\'on--Zygmund type operator $T$ to be bounded on $H^p_L(\mathbb{R}^n)$, where $L=-\Delta+V$ is a Schr\"odinger operator with $V \in RH_\sigma$ for some $\sigma\geq n/2$. The conditions are phrased as conditions on the object $T^*1$. As a consequence we also obtain the criterion for such operators $T$ to be bounded on $BMO_L^\alpha(\mathbb{R}^n)$, with conditions phrased on $T1$. We would like to describe briefly our contributions in this paper.

\begin{enumerate}[{\rm (i)}]
	\item Unlike \cite{MSTZ}, we do not assume pointwise and smoothness conditions on the associated kernel of our generalised Calder\'on--Zygmund type operators $T$. This allows us to relax the condition $ \sigma\geq n$ when considering the Riesz transforms $\nabla L^{-1/2}$,  and also allows us to consider such operators as $V^{1/2}L^{-1/2}$ and $VL^{-1}$. 
\item Our results recover those in \cite{BCFST} for the Hermite--Calder\'on--Zygmund operators,  and those in \cite{MSTZ} for their $\gamma$-Schr\"odinger--Calder\'on--Zygmund operators $T$ when $\gamma=0$.

\item The result  for boundedness on Hardy spaces (Theorem \ref{Th: CZO}) is new in the literature.

\item To prove the boundedness on Hardy spaces, we introduce an $L$-molecule satisfying size and weak cancellation condition, which is different from the $L$-molecules in the direction of work in \cite{ADM,DY,HLMMY}. Then we establish the molecular characterization of Hardy spaces.
\end{enumerate}

\subsection{Main results}

In the sequel we set $L$ as in \eqref{L}. 

The critical radius function (introduced by Shen \cite{Sh}) associated to the potential $V\in RH_\sigma$ with $\sigma\geq n/2$ is defined by
\begin{equation}\label{criticalfunction}
\rho(x)=\sup\Big\{r>0: \f{1}{r^{n-2}}\int_{B(x,r)}V(y)dy\leq 1\Big\}.
\end{equation}
As an example for the harmonic oscillator with $V(x)=|x|^2$, we have $\rho(x)\sim (1+|x|)^{-1}$. 

We also set $\sigma_0:=2-n/\sigma$, a constant which will play a key role in this article. Note that $0<\sigma_0\le 1$ precisely when $\f{n}{2}<\sigma \le n$.

We now introduce generalized Calder\'on--Zygmund type operators with respect to $L$ defined in \eqref{L} as follows.
\begin{Definition}\label{def: GCZO}
	Let $\gamma>0$, $1<\theta<\vc$ and $\theta'$ be the conjugate of $\theta$. We say that $T\in GCZK_\rho(\gamma,\theta)$ if $T$ has an associated kernel $K(x,y)$ satisfying the following estimates:
\begin{enumerate}[$(i)$]
\item For each $N>0$ there is a constant $C_N>0$ such that
	\begin{equation}\label{cond1}
	\ContainC{\int_{R<|x-x_B|<2R}\AbbsA{K(x,y)}^\theta dx}^{1/\theta}\le C_N R^{-n/\theta'}\ContainC{\frac{\rho(x_B)}{R}}^{N}
	\end{equation}
for all $y\in B(x_B,\rho(x_B))$ and all $R>2\rho(x_B)$.

\item There is a constant $C>0$ such that
    \begin{equation}\label{cond2}
		 \ContainC{\int_{2^kr_B<|x-x_B|<2^{k+1}r_B}\big|K(x,y)-K(x,x_B)\big|^\theta dx}^{1/\theta} \le C2^{-k\gamma}|2^kB|^{-1/\theta'}
	\end{equation}
	for all balls $B=B(x_B,r_B)$, all $y\in B$ and $k\geq 1$.

\end{enumerate}
We say that $T\in GCZO_\rho(\gamma,\theta)$ if $T\in GCZK_\rho(\gamma,\theta)$ and $T$ is bounded on $L^\theta(\RR{n})$.
\end{Definition} 
Note that the condition \eqref{cond2} implies the standard H\"omander condition \eqref{Hormander}, and
therefore, if $T\in GCZO_\rho(\gamma,\theta)$ for some $\gamma$ and $\theta$, then $T$ is of weak type $(1,1)$ and hence is bounded on $L^p$ for all $1<p\leq \theta$. 

We point out that the Hermite--Calder\'on--Zygmund operators of \cite{BCFST}  and the $\gamma$-Schr\"odinger--Calder\'on--Zygmund operators $T$ when $\gamma=0$ of \cite{MSTZ} belong to $GCZO_\rho(\delta,\theta)$ for certain $\delta$ and any $1<\theta<\infty$.

It is well known that in the classical situation (see \cite{HHL} for example) if  $T\in CZO$ then $T$ is bounded on the Lipschitz $\Lip{\alpha}$  for $0<\alpha <\gamma \le 1$ if and only if $T1$ is constant (we note that the Lipschitz spaces $\Lip{\alpha}$ coincide with the Campanato spaces $BMO^\alpha$ \cite{Cam}). However,
for Calder\'on--Zygmund type operators $T$ with respect to  Schrodinger operators $L$, there exist certain operators $T$ for which $T1$ or $T^*1$ is non-constant. Notable examples are the Riesz transforms $T=\nabla L^{-1/2}$.

Our main result is the following $T1$ type theorem for  $T\in GCZO_\rho(\gamma,\theta)$ to be bounded on Hardy spaces $H^p_L(\mathbb{R}^n)$ associated with $L$ defined in \eqref{L}. For the precise definition and the properties of  $H^p_L(\mathbb{R}^n)$ we refer to Section \ref{sec: Hardy}.
\begin{Theorem}\label{Th: CZO}
Let $T\in GCZO_\rho(\gamma,\theta)$ for some $0<\gamma<\sigma_0$, where $\sigma_0:=2-n/\sigma$. Then: 

\smallskip

\begin{enumerate}[\upshape(a)]
	\item $T$ is bounded on $H^1_L(\mathbb{R}^n)$ if and only if $T^*1$ satisfies
		\begin{align*}
		\log\ContainC{\frac{\rho(x_B)}{r_B}}\Big(\f{1}{|B|}\int_B\AbbsA{T^*1(y)-(T^*1)_B}^{\theta'}\,dy\Big)^{1/\theta'} \le C
	\end{align*}
	for every ball $B$ with $r_B\le\tfrac{1}{2}\rho(x_B)$.

\smallskip

	\item If $\frac{n}{n+\gamma}<p<1$, then
$T$ is bounded on $H^p_L(\mathbb{R}^n)$ if and only if $T^*1$ satisfies
\begin{align*}
		\ContainC{\frac{\rho(x_B)}{r_B}}^{n(1/p-1)}\Big(\f{1}{|B|}\int_B\AbbsA{T^*1(y)-(T^*1)_B}^{\theta'}\,dy\Big)^{1/\theta'} \le C
	\end{align*}
	for every ball $B$ with $r_B\le\tfrac{1}{2}\rho(x_B)$.

\smallskip

	\item 	If $\frac{n}{n+\gamma}<p\le 1$, then $T$ is bounded from $H^p_L(\mathbb{R}^n)$ to the classical Hardy space $H^p(\mathbb{R}^n)$ if and only if $T^*1=0$.
	\end{enumerate}
\end{Theorem}
\noindent Note that for $\f{n}{n+\sigma_0\wedge 1}<p\le 1$, the cancellation condition for atoms in Definition \ref{def: atoms} imply that the classical Hardy spaces $H^p(\mathbb{R}^n)$ are strictly contained in $H^p_L(\mathbb{R}^n)$, and thus Theorem \ref{Th: CZO} also gives boundedness from $H^p(\mathbb{R}^n)$ into $H^p_L(\RR{n})$ for (a), (b), and into $H^p(\RR{n})$ for (c).

The strategy of our proof of Theorem \ref{Th: CZO} proceeds in two steps. We firstly characterize $H^p_L(\RR{n})$ in terms of molecules associated with $L$ that have certain size and cancellation conditions (different to the $L$-molecules in the direction of work in \cite{ADM,DY,HLMMY}). See Definitions \ref{defn3.1} and \ref{defn3.2}. Secondly we show that the operators  satisfying the conditions in Theorem \ref{Th: CZO} map atoms into molecules, which yields their boundedness on $H^p_L(\RR{n})$.

As a consequence of Theorem \ref{Th: CZO} and the duality of the Hardy space $H_L^p(\mathbb{R}^n)$ with BMO type spaces (also known as  the Campanato space) $BMO_L^\alpha(\mathbb{R}^n)$, we obtain directly a $T1$ criterion for  $BMO_L^\alpha(\mathbb{R}^n)$  which extends the results of \cite{BCFST,MSTZ} to a more general setting. For the precise definition and the properties of  $BMO_L^\alpha(\mathbb{R}^n)$ we refer to Section \ref{sec: BMO}.

\begin{Definition}
	Let $\gamma>0$, $1<\theta<\vc$ and $\theta'$ be the conjugate of $\theta$. We say that $T\in GCZK^*_\rho(\gamma,\theta')$ if $T$ has an associated kernel $K(x,y)$ satisfying the following estimates:
	\begin{enumerate}[$(i)'$]
		\item For each $N>0$ there is a constant $C_N>0$ such that
		\begin{equation}\label{cond3}
		\ContainC{\int_{R<|y-x_B|<2R}\AbbsA{K(x,y)}^{\theta'} dy}^{1/\theta'}\le C_N R^{-n/\theta}\ContainC{\frac{\rho(x_B)}{R}}^{N}
		\end{equation}
		for all $x\in B(x_B,\rho(x_B))$ and all $R>2\rho(x_B)$.
		
		\item There are constants $0<\gamma\leq 1$ and $C>0$ such that
		\begin{equation}\label{cond4}
		\ContainC{\int_{2^kr_B<|y-x_B|<2^{k+1}r_B}\big|K(x,y)-K(x_B,y)\big|^{\theta'} dy}^{1/\theta'} \le C2^{-k\gamma}|2^kB|^{-1/\theta}
		\end{equation}
		for all balls $B=B(x_B,r_B)$, all $x\in B$ and $k\geq 1$.
	\end{enumerate}
We say that $T\in GCZO^*_\rho(\gamma,\theta')$ if $T\in GCZK^*_\rho(\gamma,\theta')$ and $T$ is bounded on $L^{\theta'}(\RR{n})$. \end{Definition}
We wish to make two observations. Firstly, whereas  Definition \ref{def: GCZO} specifies a certain regularity in the second variable, the requirement here is in the first variable.   Secondly  if $T$ belongs to $GCZO^*_\rho(\gamma,\theta')$ for some $\gamma$ and $\theta$, then $T$ is automatically bounded on $L^p$ for all $\theta'\le p <\infty$.

\begin{Theorem}\label{thm2}

Let $T\in GCZO^*_\rho(\gamma,\theta')$ for some $0<\gamma<\sigma_0$, where $\sigma_0:=2-n/\sigma$. Then:

\begin{enumerate}[\upshape(a)$'$]
	\item $T$ is bounded on $BMO_L(\mathbb{R}^n)$ if and only if $T1$ satisfies
	\begin{align*}
	\log\ContainC{\frac{\rho(x_B)}{r_B}}\Big(\f{1}{|B|}\int_B\AbbsA{T1(y)-(T1)_B}^{\theta}\,dy\Big)^{1/\theta} \le C
	\end{align*}
	for every ball $B$ with $r_B\le\tfrac{1}{2}\rho(x_B)$.
	\item If $0<\alpha<\gamma$, then
	$T$ is bounded on $BMO_L^\alpha(\mathbb{R}^n)$ if and only if $T1$ satisfies
	\begin{align*}
	\ContainC{\frac{\rho(x_B)}{r_B}}^{\alpha}\Big(\f{1}{|B|}\int_B\AbbsA{T1(y)-(T1)_B}^{\theta}\,dy\Big)^{1/\theta} \le C
	\end{align*}
	for every ball $B$ with $r_B\le\tfrac{1}{2}\rho(x_B)$.
	\item 	If $0<\alpha<\gamma$ then $T$ is bounded from $BMO^\alpha(\mathbb{R}^n)$ into $BMO_L^\alpha(\mathbb{R}^n)$ if and only if $T1=0$.
\end{enumerate}

\end{Theorem}

\subsection{Applications}
We now present some applications to singular integrals related to $L$. The precise definitions of the listed operators  will be provided in Section \ref{sec: misc ops}. %Firstly we give applications to some operators considered in \cite{MSTZ}.

%\begin{Theorem}\label{thm3 appl1}
%For $ {n\over n+\sigma_0\wedge 1}<p\leq 1$,
 %the maximal operators associated with the heat 
%semigroup and with the generalized Poisson operators, the Littlewood--Paley g-functions given in terms of the heat and the Poisson semigroups, and the Laplace transform type multipliers are bounded on $H^p_L(\mathbb{R}^n)$. As a consequence, for $0\leq \alpha< \sigma_0\wedge 1$, these operators are bounded on   $BMO^\alpha_L(\mathbb{R}^n)$.
%\end{Theorem}
%\noindent We point out that the above results recover those listed in Theorem 1.3 in \cite{MSTZ}. 
\begin{Theorem}\label{thm3 appl1}
For $ {n\over n+\sigma_0\wedge 1}<p\leq 1$,
  the Laplace transform type multipliers $m(L)$ are bounded on $H^p_L(\mathbb{R}^n)$. As a consequence, for $0\leq \alpha< \sigma_0\wedge 1$, these operators are bounded on   $BMO^\alpha_L(\mathbb{R}^n)$.
\end{Theorem}
\noindent We point out that the above result recovers the $BMO^\alpha_L$ result in Theorem 1.3 in \cite{MSTZ}, while the hardy space result is new. We also mention that using the vector-valued approach in \cite{MSTZ}, we can also apply Theorem \ref{thm2} to recover boundedness on $BMO^\alpha_L$ of the other operators listed in \cite{MSTZ} Theorem 1.3, namely the maximal operators and Littlewood--Paley $g$-functions associated with the heat and Poisson semigroups. 

\bigskip

Next we have the following result for the Riesz transforms $\RT=\nabla L^{-1/2}$ and $\RRT = \nabla^2 L^{-1}$.
 \begin{Theorem}\label{thm4 appl2}
The Riesz transforms $\RT$ and $\RRT$ are bounded from $H^p_L(\mathbb{R}^n)$ into $H^p(\mathbb{R}^n)$ for all $\f{n}{n+\sigma_0\wedge 1}<p\leq 1$. As a consequence $\RT^*$ and $\RRT^*$ are bounded from $BMO^\alpha(\mathbb{R}^n)$ to $BMO^\alpha_L(\mathbb{R}^n)$ for $0\leq \alpha<\sigma_0\wedge 1$.
\end{Theorem}
\noindent The results in Theorem \ref{thm4 appl2} are not new. Indeed it is known that both $\RT$ and $\RRT$ are bounded from $H^p_L$ into $L^p$ for all $0<p\le 1$ and from $H^p_L$ into $H^p$ for all $\f{n}{n+1}<p\le 1$ (see \cite{HLMMY,JY,FKL}).

\bigskip

We also apply our results to Riesz transforms induced by the potential $V$ such as $V^{1/2}L^{-1/2}$ and $VL^{-1}$, which  were earlier  shown by Shen \cite{Sh} to be $L^p$-bounded for $1\le p \le 2\sigma$ and $1\le p\le \sigma$ respectively. 
While such operators are not of Calder\'on--Zygmund type, we will see that they nonetheless fall into the scope of Theorems \ref{Th: CZO} and \ref{thm2}.

In fact we shall  consider their generalizations $V^sL^{-s}$, for $0<s\le 1$,  which are   $L^p$ bounded for $1<p<\f{\sigma}{s}$ (see \cite{Su}). 
 \begin{Theorem}\label{thm5 appl3}
For each $0<s\le 1$ the operators $V^s L^{-s}$ are bounded on $H^p_L(\mathbb{R}^n)$ for each $\f{n}{n+s\sigma_0\wedge 1}<p\le 1$. 
	As a consequence the operators $(V^sL^{-s})^*$ are bounded on $BMO_L^\alpha$ for each $0\le \alpha< s\sigma_0\wedge 1$.

\end{Theorem}
\noindent 
The results in Theorem \ref{thm5 appl3} are new, although the cases $s=\f{1}{2}$ and $s=1$ are known to map $H^p_L$ into $L^p$ for $\f{n}{n+1}<p\le 1$ (see \cite{FKL}).

\bigskip

 One may ask which operators $T$ and their adjoints $T^*$ are both bounded on $H^p_L$ (and consequently $BMO^\alpha_L$)? Applying Theorems \ref{Th: CZO} and \ref{thm2} would require that they be members of both $GCZO$ and $GCZO^*$, and recall from earlier remarks that this  imposes the $L^p$ boundedness of $T$ for $p$ close to both 1 and $\infty$. This can be guarunteed for example when $T$ is a Calder\'on--Zygmund operator, which is true of $\RT$ when $\sigma\ge n$, and of $\RRT$ when $V$ is a non-negative polynomial \cite{Zh}. In our final application, we show that with sufficient regularity on $V$, the operators $V^sL^{-s}$ and their adjoints $L^{-s}V^s$ both fall into the scope of Theorem \ref{Th: CZO}. 

\begin{Theorem}\label{thm6 appl4}
Suppose that $V\in \RH_\infty$ and that for some $C>0$
\begin{align}\label{V cond1}
|\nabla V(x)| \le C \rho(x)^{-3} \quad \text{a.e. x} 
\end{align}
Then for each $0<s\le 1$, the operator $L^{-s}V^s$ is bounded from $H^p_L$ into $H^p_L$ for $\f{n}{n+2s\wedge 1} < p\le 1$. 
As a consequence $V^sL^{-s}$ is bounded from $BMO^\alpha_L$ into $BMO^\alpha_L$ for all $0\le \alpha  < 2s\wedge 1$. 
\end{Theorem}
\noindent
The condition $V\in\RH_\infty$ ensures that both $V^sL^{-s}$ and $L^{-s}V^s$ are $L^p$ bounded for all $1<p<\infty$, while \eqref{V cond1} furnish sufficient smoothness for the conditions of Theorems \ref{Th: CZO} and \ref{thm2} to hold. Examples of $V$ satisfying the conditions of Theorem \ref{thm6 appl4} are non-negative polynomials and in particular include the harmonic oscillator $V(x)=|x|^2$.

\bigskip

This paper is organised as follows. In Section 3 we recall the Hardy and BMO type spaces associated to Schr\"odinger operator $L$, and  introduce a new molecular decomposition for the Hardy spaces. In Section 4 we provide the proof of the $T1$ criterions Theorems \ref{Th: CZO} and \ref{thm2} for Hardy and BMO type spaces respectively. Finally in Section 5 we give applications of the $T1$ criterion by proving Theorems \ref{thm3 appl1}--\ref{thm6 appl4}. 

Throughout the paper, we always use $C$ and $c$ to denote positive constants that are independent of the main parameters involved but whose values may differ from line to line. We will write $A\lesi B$ if there is a universal constant $C$ so that $A\leq CB$ and $A\sim B$ if $A\lesi B$ and $B\lesi A$. Given a ball $B$ we refer to the ball $B(x_B,r_B)$ with centre $x_B$ and radius $r_B$. We also denote by $\rho_B:=\rho(x_B)$. The notation 
$$\barint_B f= \f{1}{|B|}\int_B f$$
refers to the average of $f$ on $B$. The expression $a\wedge b$ denotes the minimum of $a$ and $b$. Given a ball $B$, the set $U_j(B)$ denotes $2^jB\backslash 2^{j-1}B$ for $j\ge 1$ and denotes $B$ if $j=0$.

\section{Preliminaries}
In this section we recall the well-known heat kernel upper bounds for the Schr\"odinger operator as well as properties for $V$ and its critical radius function $\rho$ as defined in \eqref{criticalfunction}.

The following estimates on the heat kernel of $L$ are well known.
\begin{Proposition}{\upshape(\cite{DZ2,DZ3})}\label{prop-kernelestimates}
	Let $L=-\Delta+V$ with $V\in \RH_\sigma$ for some $\sigma\ge n/2$. Then for each $N>0$ there exists $C_N>0$ such that
	\begin{align}\label{hk bound}
		p_t(x,y)\le C_N\frac{e^{-|x-y|^2/ct}}{t^{n/2}} \ContainC{1+\frac{\sqrt{t}}{\rho(x)}+\frac{\sqrt{t}}{\rho(y)}}^{-N}
	\end{align}
	and
	\begin{align}\label{hk holder}
		|p_t(x,y)-p_t(x',y)|\le C_N\ContainC{\frac{|x-x'|}{\sqrt{t}}}^{\sigma_1} \frac{e^{-|x-y|^2/ct}}{t^{n/2}} \ContainC{1+\frac{\sqrt{t}}{\rho(x)}+\frac{\sqrt{t}}{\rho(y)}}^{-N}
	\end{align}
	whenever $|x-x'|\le \sqrt{t}$ and for any $0<\sigma_1<\sigma_0$.
\end{Proposition}

For $\sigma>1$, the class of locally integrable functions satisfying \eqref{RH} will be denoted $\RH_\sigma$. For $\sigma=\infty$, the left hand side of \eqref{RH} is replaced by the essential supremum over $B$. It is well known that elements of $\RH_\sigma$ are doubling measures, and that $\RH_\sigma\subset \RH_{\sigma'}$ whenever $\sigma'<\sigma$ . 

We list but do not prove the following properties of the critical function $\rho$ in \cite{Sh}.
\begin{Lemma}\label{Lem1: rho}
Let $\rho$ be the critical radius function associated with $L$ defined in \eqref{criticalfunction}. Then we have:
\begin{enumerate}[{\upshape (i)}]
\item There exist positive constants $k_0 \ge 1$ and $C_0>0$ so that
$$
C_0^{-1}[\rho(x)]^{1+k_0}[\rho(x)+|x-y|]^{-k_0}\leq \rho(y)\leq C_0[\rho(x)]^{1/(1+k_0)}[\rho(x)+|x-y|]^{k_0/(1+k_0)},
$$
for all $x,y\in \mathbb{R}^n$.

In particular for any ball $B$, and any $x,y\in B$ then $\rho(x)\le C_0^2 \bigl(1+\f{r_B}{\rho_B}\bigr)^2\rho(y)$.
\item There exists $C>0$ so that
$$
\f{1}{r^{n-2}}\int_{B(x,r)}V(y)dy\leq C\Big(\f{r}{R}\Big)^{\sigma_0}\f{1}{R^{n-2}}\int_{B(x,R)}V(y)dy
$$
for all $x\in M$ and $R>r>0$. 
\item For any $x\in M$, we have
$$
\f{1}{\rho(x)^{n-2}}\int_{B(x,\rho(x))}V(y)dy=1.
$$
\item There exists $C>0$ so that for any $r>\rho(x)$ 
$$ r^2 \barint_{B(x,\rho(x)} V(y)\,dy \le C \Bigl(\f{r}{\rho(x)}\Bigr)^{n_0-n+2}
$$
where $n_0$ is the doubling order of $V$. That is, $\int_{2B} V\lesssim 2^{n_0} \int_B V$ for any ball $B$. 
\end{enumerate}
\end{Lemma}
\begin{Remark}\label{rem: V estimate}
It follows from Lemma \ref{Lem1: rho} (ii) and (iii) that for any ball $B$,
\begin{align*}
	r_B^2\barint_{B}V(y)\,dy \lesssim \left\lbrace\begin{array}{ll}
		\Bigl(\dfrac{r_B}{\rho_B}\Bigr)^{\sigma_0} \qquad & r_B\le \rho_B\\
		\Bigl(\dfrac{r_B}{\rho_B}\Bigr)^{n_0+2-n} \qquad & r_B>\rho_B
		\end{array}
		\right. 
	\end{align*}

\end{Remark}

\begin{Lemma}[\cite{DGMTZ}]\label{Lem2: rho}
Let $\rho$ be a critical function associated to Schr\"odinger operators $L=-\Delta+V$. Then there exists a sequence of points $\{x_\alpha\}_{\alpha\in \mathcal{I}}\subset \RR{n}$ and a family of functions $\{\psi_\alpha\}_{\alpha\in \mathcal{I}}$ satisfying for some $C>0$
\begin{enumerate}[{\rm (i)}]
\item $\bigcup_\alpha B(x_\alpha, \rho(x_\alpha)) = \RR{n}$.
\item For every $\lambda \geq 1$ there exist constants $C$ and $N_1$ such that $\sum_\alpha \chi_{B(x_\alpha, \rho(x_\alpha))}\leq C\lambda^{N_1}$.
\item ${\rm supp}\, \psi\subset B(x_\alpha, \rho(x_\alpha)/2)$ and $0\leq \psi_\alpha(x)\leq 1$ for all $x\in \RR{n}$;
\item $|\psi_\alpha(x)-\psi_\alpha(y)|\leq C |x-y|/\rho(x_\alpha)$;
\item $\sum_{\alpha}\psi_\alpha(x)=1$ for all $x\in \RR{n}$.
\end{enumerate}
\end{Lemma}

\section{Hardy and Campanato spaces associated with Schr\"odinger operator}

In this section we recall the definition of Hardy space $H^p_L(\mathbb{R}^n)$ associated to  $L$ in terms of the maximal operator and of  atoms. Then we introduce a new kind of molecule for these $H^p_L(\mathbb{R}^n)$ in terms of size condition and weak cancellation condition,  and then we provide the molecule  characterisation for $H^p_L(\mathbb{R}^n)$. We also recall the BMO type space associated to  $L$, and note that it is the dual of $H^p_L(\mathbb{R}^n)$.

\subsection{Hardy spaces}\label{sec: Hardy}
We now recall some properties related to the atomic decomposition of Hardy spaces associated to Schr\"odinger operators. For further details on the theory of Hardy spaces associated to Schr\"odinger operators, we refer the reader to \cite{DZ1, DZ2, DZ3, JYY, YYd, YZ} and the references therein.

We first define the maximal operator associated to the heat semigroup:
$$ \MM_L f(x):=\sup_{t>0}\AbbsA{e^{-tL}f(x)} $$
For $0<p\le 1$ we denote by $L^p_b(\RR{n})$ the set of all $L^p$-functions with bounded support. We then set
$$ \mathfrak{S}_p(\RR{n}):= \bigl\{f: f\in L^s_b(\RR{n}) \;\text{for every}\; s\in [1,\infty]\bigr\}$$
Following \cite{DZ2} we define
\begin{Definition}[Hardy spaces]
	For $p\in (0,1]$, the Hardy space $H^p_L(\RR{n})$ is defined as the completion of
	$$\mathbb{H}^p_L:=\big\{f\in\mathfrak{S}_p: \MM_L f\in L^p\big\}$$
	in the quasi norm $\NormA{f}_{H^p_L} := \NormA{\MM_Lf}_p$.
\end{Definition}

\begin{Definition}\label{def: atoms}
	Let $0<p\le 1$ and $1<q\le \infty$. A function $a$ is called an $(p,q)_L$-atom for $L$ associated with a ball $B$
	\begin{enumerate}[\upshape(i)]
		\item $r_B\le  \rho_B$
		\item $\supp a\subset B$
		\item $\NormA{a}_q \le\AbbsA{B}^{1/q-1/p}$
		\item $\int a(x)\,dx=0$ whenever  $r_B< \rho_B/4$
	\end{enumerate}
\end{Definition}

Let $\frac{n}{n+\sigma_0}<p\le 1$ and $1<q\le \infty$. We then define the atomic Hardy spaces $H^{p,q}_{L,at}(\RR{n})$ as the completion of
\begin{equation}\label{eq-atomicHardy}
\mathbb{H}^{p,q}_{L,at}(\RR{n})=\{f: f=\sum_{j=1}^\vc \lambda_j a_j \ \text{in $L^2$},\  a_j \ \text{is an $(p,q)_L$-atom and  $\sum_j |\lambda_j|^p<\infty$}\}
\end{equation}
with respect to the norm
$$
\|f\|_{H^{p,q}_{L,at}(\RR{n})}=\inf\Big\{\Big[\sum_j |\lambda_j|^p\Big]^{1/p}: f=\sum_{j=1}^\vc \lambda_j a_j\Big\}.
$$

We also define the Hardy spaces in terms of finite atoms.
\begin{Definition}
We define $H^{p,q}_{L,at, {\rm fin}}(\RR{n})$ as the set of all functions $f=\sum_{j=1}^N \lambda_j a_j$, where $a_j$ is an $(p,q)_L$-atom if $q<\vc$ and continuous $(p,q)_L$-atom if $q=\vc$. For $f\in H^{p,q}_{L,at, {\rm fin}}(\RR{n})$, we define $\|f\|_{H^{p,q}_{L,at, {\rm fin}}(\RR{n})}$ similarly to $\|f\|_{H^{p,q}_{L,at}(\RR{n})}$, but the infimum is taken over finite linear decomposition of $(p,q)_L$-atoms.
\end{Definition}

We have the following result.
\begin{Proposition}
Let $\frac{n}{n+\sigma_0\wedge 1}<p\le 1$ and $1<q\le \infty$. Then we have the spaces $H^p_L(\RR{n})$ and $H^{p,q}_{L,at}(\RR{n})$ are coincide with equivalent norms.
\end{Proposition}
\begin{proof}
It was proved in \cite{DZ2} that $H^p_L(\RR{n})\equiv H^{p}_{L,at,\vc}(\RR{n})$. From definition of $H^{p,q}_{L,at}(\RR{n})$, we have $H^{p}_{L,at,\vc}(\RR{n})\hookrightarrow H^{p,q}_{L,at}(\RR{n})$. On the other hand, by a standard argument, see for example \cite{DZ2, DZ3}, we can prove that $H^{p,q}_{L,at}(\RR{n})\hookrightarrow H^p_L(\RR{n})$. This implies that $H^p_L(\RR{n})$ and $H^{p,q}_{L,at}(\RR{n})$ are coincide with equivalent norms.
\end{proof}

We now prove the following result.

\begin{Proposition}\label{prop equiv finite atom}
Let $\frac{n}{n+\sigma_0\wedge 1}<p\le 1$ and  $1<q\le \infty$. Then the norms $\|\cdot\|_{H^{p,q}_{L,at, {\rm fin}}(\RR{n})}$ and $\|\cdot\|_{H^{p,q}_{L,at}(\RR{n})}$ are equivalent in $H^{p,q}_{L,at, {\rm fin}}(\RR{n})$.
\end{Proposition}
\begin{proof}
Let $f\in H^{p,q}_{L,at, {\rm fin}}(\RR{n})$. Obviously, we have
$$
\|f\|_{H^{p,q}_{L,at}(\RR{n})} \leq \|f\|_{H^{p,q}_{L,at, {\rm fin}}(\RR{n})}.
$$
Hence, it suffices to prove the converse inequality. Indeed, we first note that $f=\sum_{\alpha\in \mathcal{I}_f}\psi_\alpha f$ where $\mathcal{I}_f=\{\alpha: B_\alpha\cap {\rm supp}\, f\neq \emptyset\}$. Since ${\rm supp}\, f$ is bounded, from Lemma \ref{Lem1: rho}, the set $\mathcal{I}_f$ is finite. Hence,
$$
\|f\|_{H^{p,q}_{L,at, \epsilon, {\rm fin}}(M)}\leq \sum_{\alpha\in \mathcal{I}_f}\|\psi_\alpha f\|_{H^{p,q}_{L,at, \epsilon, {\rm fin}}(M)}.
$$
From the theory of local Hardy spaces in Theorem 3.12 and Theorem 6.2 in \cite{YY} (see also \cite{YY1}), we also get that
$$
\sum_{\alpha\in \mathcal{I}_f}\|\psi_\alpha f\|_{H^{p,q}_{L,at, {\rm fin}}(\mathbb{R}^n)}\lesi \sum_{\alpha\in \mathcal{I}_f}\Big\|\sup_{0<t<[\rho(x_\alpha)]^2}|e^{-t\Delta}\psi_\alpha f|\,\Big\|_{L^p(\mathbb{R}^n)}.
$$
We now just follows the argument as in \cite[p. 53]{DZ2} to conclude that
$$
\|f\|_{H^{p,q}_{L,at,{\rm fin}}(M)}\lesi \|\mathcal{M}_Lf\|_{L^p(M)}.
$$
This completes our proof.
\end{proof}
\subsection{Molecular characterizations}
In this section we introduce a new kind of molecule, and show that the Hardy spaces $H^p_L$ can be characterized by such molecules. 
\begin{Definition}[Molecules for $p=1$]\label{defn3.1}
	Let $1<q\le\infty$. A function $m$ is called an $(1,q,\beta)_L$-molecule for $H^1_L$ associated to the ball $B$ if for some $\beta>0$
	\begin{enumerate}[\upshape(a)]
		\item $r_B\le\rho_B$
		\item $\displaystyle\NormA{m}_{L^q(U_j(B))} \le 2^{- j \beta}\AbbsA{2^jB}^{1/q-1}$ for all $j=0,1,2,\dots$
		\item $\displaystyle \Big|\int_{\RR{n}} m(x)\,dx\Big|\le \frac{1}{\log(\rho_B/r_B)}$.
	\end{enumerate}
An $(1,q,\beta)_L$-molecule associated to the ball $B$ supported in $B$ is called an $(1,q)_{\log}$-atom.
\end{Definition}

\begin{Definition}[Molecules for $p<1$]\label{defn3.2}
	Let $p\in (0,1)$ and $1<q\le \infty$. A function $m$ is called a $(p,q,\beta,\delta)_L$-molecule for $L$ associated to the ball $B$ if for some $\beta,\delta>0$
	\begin{enumerate}[\upshape(a)]
		\item $r_B\le\rho_B$
		\item $\displaystyle\NormA{m}_{L^q(U_j(B))} \le 2^{- j \beta}\AbbsA{2^jB}^{1/q-1/p}$ for all $j=0,1,2,\dots$
		\item $\displaystyle \Big|\int_{\RR{n}} m(x)\,dx\Big|\le |B|^{1-1/p}\ContainC{\frac{r_B}{\rho_B}}^\delta$.
	\end{enumerate}
A $(p,q,\beta,\delta)_L$-molecule associated to the ball $B$ supported in $B$ is called a $(p,q,\delta)_L$-atom.
\end{Definition}

It is easy to see that a $(p,q)_L$-atom is a multiple of a $(p,q,\beta,\delta)_L$-molecule for any $\delta>0$, $\beta>0$. The next result is an almost-orthogonality type estimate for atoms.
\begin{Lemma}\label{Lem: almost orthog}
	Let $p\in (\frac{n}{n+\sigma_0\wedge 1},1)$, $1<q\le \infty$ and $\delta>0$. Let $a$ be a $(p,q,\delta)_L$-atom for $L$ associated to a ball $B$ as in Definition \ref{defn3.2}. Then for any $\nu < \min \{\sigma_0, \delta\}$, there exists $C>0$ so that
	\begin{align*}
		|e^{-tL}a(x)|\leq C			\dfrac{r_B^{\nu}}{|x-x_B|^{n+\nu}}|B|^{1-1/p},
			\end{align*}
for all $x\in \mathbb{R}^n\backslash 4B.$
\end{Lemma}
\begin{proof}
	We write
	\begin{align*}
		e^{-tL}a(x)
		= \int_B [p_t(x,y)-p_t(x,x_B)]a(y)\,dy + p_t(x,x_B)\int_B a(y)\,dy =:I+II
	\end{align*}
	Now from the bounds on the heat kernel, and the cancellation for $a$ we have
	\begin{align*}
		II
		&\le \big|p_t(x,x_B)\big| \;\bigg|\int a(y)\,dy\bigg| \\
		&\lesssim
		\frac{e^{-|x-x_B|^2/ct}}{t^{n/2}} \ContainC{1+\frac{\sqrt{t}}{\rho(x)}+\frac{\sqrt{t}}{\rho_B}}^{-N} |B|^{1-1/p}\ContainC{\frac{r_B}{\rho_B}}^\nu \\
		&\lesssim
		\dfrac{t^{\nu/2}}{(\sqrt{t}+|x-x_B|)^{n+\nu}} \ContainC{1+\frac{\sqrt{t}}{\rho(x)}+\frac{\sqrt{t}}{\rho_B}}^{-N} |B|^{1-1/p}\ContainC{\frac{r_B}{\rho_B}}^\nu\\
		&\lesssim \dfrac{r_B^{\nu}}{|x-x_B|^{n+\nu}}|B|^{1-1/p}
	\end{align*}
	by choosing $N=\nu$.
	
	Next by using Proposition \ref{prop-kernelestimates} we write
	\begin{align*}
		I &\lesssim \int_{B}\Big(\f{|y-x_B|}{|x-y|}\Big)^{\nu}\frac{e^{-|x-x_B|^2/ct}}{t^{n/2}} |a(y)|dy\\
&\lesssim \int_{B}\Big(\f{r_B}{|x-x_B|}\Big)^{\nu}\frac{e^{-|x-x_B|^2/ct}}{t^{n/2}}|a(y)|dy\\
&\lesssim \Big(\f{r_B}{|x-x_B|}\Big)^{\nu}\frac{e^{-|x-x_B|^2/ct}}{t^{n/2}} |B|^{1-1/p}\\
&\lesssim \f{r^\nu_B}{|x-x_B|^{n+\nu}}|B|^{1-1/p}.
	\end{align*}
\end{proof}

\begin{Lemma}[Molecules are in $H^p_L$]\label{Lem: molecules in Hp}
 Let $p\in (\frac{n}{n+\sigma_0\wedge 1},1]$ and $1<q\le \infty$. If $m$ is a $(p,q,\beta,\delta)_L$ molecule associated to a ball $B$ with $q>1$ and $\beta>0$ and $\delta>n(1/p-1)$ then $m$ is in $H^p_L$.
\end{Lemma}
\begin{proof}
	
We divide into two cases:

\noindent {\bf Case 1: $p<1$}

We wish to show
	$$ \NormA{\MM_L(m)}_{L^p}\le C.$$
	To do this we set
for $j\ge 0$, $\alpha_j = \int_{U_j(B)} m(x)dx$ and $\chi_j=\f{1}{|U_j(B)|}\chi_{U_j(B)}$. Then we define
$$
a_j(x)=m(x)\chi_{U_j(B)}(x)-\alpha_j \chi_j(x).
$$
If we set $N_j=\sum_{k=j}^\vc \alpha_k$, then we have
\begin{equation}\label{eq-m}
\begin{aligned}
m(x)&=\sum_{j=0}^\vc a_j(x) +\sum_{j=0}^\vc N_{j+1}(\chi_{j+1}(x)-\chi_j(x))+ \chi_0(x)\int m(y)dy\\
&=\sum_{j=0}^\vc a_j(x)+ \sum_{j=0}^\vc b_j(x) + a(x),
\end{aligned}
\end{equation}
which implies
$$
\begin{aligned}
\NormA{\MM_L(m)}^p_{L^p}&\leq \sum_{j=0}^\vc \NormA{\MM_L(a_j)}^p_{L^p}+ \sum_{j=0}^\vc\NormA{\MM_L(b_j)}^p_{L^p} + \NormA{\MM_L(a)}^p_{L^p}\\
&\leq I_1 + I_2 + I_3.
\end{aligned}
$$
We now take care of the terms in $I_1$ first. We note that
\begin{equation}\label{eq1-a_j}
{\rm supp}\,a_j\subset 2^jB, \int a_j =0 \quad \text{and} \quad \|a_j\|_{L^q}\leq C 2^{-j\beta}|2^jB|^{1/q-1/p}.
\end{equation}
Hence, for $x\in \mathbb{R}^n\backslash 2^{j+2}B$ we have
\begin{equation}\label{eq2-a_j}
\begin{aligned}
|e^{-tL}a_j(x)|&=\Big|\int_{2^jB}[p_t(x,y)-p_t(x,x_B)]a_j(y)dy\Big|\\
&\lesssim \int_{2^jB}\Big(\f{|y-x_B|}{|x-y|}\Big)^{\nu}\frac{e^{-|x-x_B|^2/ct}}{t^{n/2}} |a(y)|dy\\
&\lesssim \int_{2^jB}\Big(\f{r_B}{|x-x_B|}\Big)^{\nu}\frac{e^{-|x-x_B|^2/ct}}{t^{n/2}}|a(y)|dy\\
&\lesssim 2^{-j\beta}\Big(\f{2^jr_B}{|x-x_B|}\Big)^{\nu}\frac{e^{-|x-x_B|^2/ct}}{t^{n/2}} |2^jB|^{1-1/p}\\
&\lesssim 2^{-j\beta}\f{(2^jr_B)^\nu}{|x-x_B|^{n+\nu}}|2^jB|^{1-1/p},
\end{aligned}
\end{equation}
where $n(1/p-1)<\nu<\min\{\sigma_0,\delta\}$.

We now observe that
$$ \NormA{\MM_L(a_j)}_{L^p}= \NormA{\MM_L(a_j)}_{L^p(2^{j+2}B)} + \NormA{\MM_L(a_j)}_{L^p(\RR{n}\backslash 2^{j+2}B)}.$$
Then by the $L^q$-boundedness of $\MM$, H\"older's inequality and \eqref{eq1-a_j}we have
	\begin{align*}
		\NormA{\MM_L(a_j)}_{L^p(2^{j+2}B)}
		&\lesssim |2^jB|^{1/p-1/q} \NormA{\MM_L(a_j)}_{L^q(2^{j+2}B)} \\
		&\lesssim |2^jB|^{1-p/q} \NormA{a_j}_{L^q} \\
		&\lesssim |2^{j}B|^{1/p-1/q} 2^{-j\beta}|2^jB|^{1/q-1/p} \le 2^{-j\beta}.
	\end{align*}
On the other hand, by \eqref{eq2-a_j},
\begin{align*}
		\NormA{\MM_L(a_j)}_{L^p(\mathbb{R}^n\backslash 2^{j+2}B)}
		&\lesssim 2^{-j\beta}|2^jB|^{1-1/p}\Big(\int_{\mathbb{R}^n\backslash 2^{j+2}B}\Big(\f{(2^jr_B)^\nu}{|x-x_B|^{n+\nu}}\Big)^p dx\Big)^{1/p}
        \lesssim 2^{-j\beta},
	\end{align*}
as long as $\nu>n(1/p-1)$.

As a consequence, $I_1\lesssim \sum_{j\geq 0}  2^{-j\beta}\leq C$.

Next we also observe that
	\begin{equation}\label{eq1-b_j}
{\rm supp}\,b_j\subset 2^{j+1}B \quad \text{and} \quad \int b_j =0.
\end{equation}
Moreover,
$$
\|b_j\|_{L^q}\leq |N_{j+1}||2^jB|^{1/q-1}.
$$
From the definition of $N_{j+1}$ and H\"older's inequality, we can get that
$$
\begin{aligned}
|N_{j+1}|&\leq \sum_{k\geq j+1}\int_{S_k(B)}|m(y)|dy\leq  \sum_{k\geq j+1}|2^kB|^{1-1/q}\|m\|_{L^q(S_k(B))}\\
&\leq \sum_{k\geq j+1}|2^kB|^{1-1/q}2^{-k\beta}|2^kB|^{1/q-1/p}:=\sum_{k\geq j+1}2^{-k\beta}|2^kB|^{1-1/p}\\
&\leq 2^{-j\beta}\sum_{k\geq j+1}2^{-(k-j)(\beta+n(1-1/p))}|2^jB|^{1-1/p}\\
&\leq C2^{-j\beta}|2^jB|^{1-1/p}.
\end{aligned}
$$
This implies that
\begin{equation}\label{eq2-b_j}
\|b_j\|_{L^q}\leq C2^{-j\beta}|2^jB|^{1/q-1/p}.
\end{equation}

At this stage, an similar argument used to estimate $I_1$, we also arrive at that $I_2\leq C$.

For the last term $I_3$, we proceed as follows:
$$
\|\MM_L(a)\|_{L^p}\leq \|\MM_L(a)\|_{L^p(4B)} +\|\MM_L(a)\|_{L^p(\mathbb{R}^n\backslash 4B)}.
$$
For the first term, using the $L^q$-boundedness of $\MM_L$ and H\"older' inequality to dominate it by
$$
\begin{aligned}
\|\MM_L(a)\|_{L^p(4B)}&\leq C|B|^{1/p-1/q}\|a\|_{L^q}\leq C|B|^{1/p-1/q}|B|^{1/q-1}\Big|\int m(y)dy\Big|\\
&\leq C|B|^{1/p-1}|B|^{1-1/p}\Big(\f{r_B}{\rho_B}\Big)^\delta \leq C.
\end{aligned}
$$
We now apply Lemma \ref{Lem: almost orthog} to see that
$$
\begin{aligned}
\|\MM_L(a)\|_{L^p(\mathbb{R}^n\backslash 4B)}&\leq C|B|^{1-1/p}\Big(\int_{\mathbb{R}^n\backslash 4B} \Big[\dfrac{r_B^{\nu}}{|x-x_B|^{n+\nu}}\Big]^p\Big)^{1/p} \leq C,
\end{aligned}
$$
provided $\nu>n(1/p-1)$.

\noindent {\bf Case 2: $p=1$}

Similarly to \eqref{eq-m} we write
$$
\begin{aligned}
m(x)&=\sum_{j=0}^\vc a_j(x) +\sum_{j=0}^\vc N_{j+1}(\chi_{j+1}(x)-\chi_j(x))+ \chi_0(x)\int m(y)dy\\
&=\sum_{j=0}^\vc a_j(x)+ \sum_{j=0}^\vc b_j(x) + a(x).
\end{aligned}
$$
The argument as in Case 1 has shown that $\sum_{j=0}^\vc a_j(x)+ \sum_{j=0}^\vc b_j(x)$ is in $H^1_L$. It remains to show that $a(x):=\chi_0(x)\int m(y)dy\in H^1_L$. By Proposition \ref{prop-CMO} we claim that
$$
\Big|\int_B a(x)\phi(x)dx\Big|\leq C\|\phi\|_{BMO_L},
$$
for all $\phi\in C^\vc(\mathbb{R}^n)$.

Indeed, we have
$$
\Big|\int_B a(x)\phi(x)dx\Big|\leq \Big|\int_B a(x)(\phi(x)-\phi_B)dx\Big| + |\phi_B|\Big|\int_B a(x)dx\Big|.
$$
By H\"older's inequality we have
$$
\begin{aligned}
 \Big|\int_B a(x)(\phi(x)-\phi_B)dx\Big|&\leq \|a\|_{L^q(B)}\Big(\int_B|\phi(x)-\phi_B|^{q'}dx\Big)^{1/q'}\\
 &\leq C |B|^{1/q-1}|B|^{1/q'}\|\phi\|_{BMO_L}:=C\|\phi\|_{BMO_L}.
\end{aligned}
$$
To dominate the second term we note that by \cite[Lemma 2]{DGMTZ}, we have
$$
|\phi_B|\lesssim \|\phi\|_{BMO_L}\log\Big(\f{\rho_B}{r_B}\Big).
$$
Inserting this into the second term to obtain that
$$
\begin{aligned}
 |\phi_B|\Big|\int_B a(x)dx\Big|&\lesssim \|\phi\|_{BMO_L}\log\Big(\f{\rho_B}{r_B}\Big) \Big|\int_B m(x)dx\Big| \lesssim \|\phi\|_{BMO_L}.
\end{aligned}
$$
This completes our proof.
\end{proof}

\begin{Proposition}[Molecular characterization]
$H^1_L(\mathbb{R}^n)$ is equivalent to the completion of
\begin{equation}\label{eq-molecularHardy 1}
\mathbb{H}^{1,q}_{L,mol}(\RR{n})=\{f: f=\sum_{j=1}^\vc \lambda_j m_j \ \text{in $L^2$},\  m_j \ \text{is an $(1,q,\beta)_L$-molecule and  $\sum_j |\lambda_j|^p<\infty$}\}
\end{equation}
with respect to the norm
$$
\|f\|_{H^{1,q}_{L,mol}(\RR{n})}=\inf\Big\{\sum_j |\lambda_j|: f=\sum_{j=1}^\vc \lambda_j m_j\Big\}.
$$

For $p\in(\frac{n}{n+\sigma_0},1)$ then $H^p_L$ is equivalent to the completion of
\begin{equation}\label{eq-molecularHardy p}
\mathbb{H}^{p,q}_{L,mol}(\RR{n})=\{f: f=\sum_{j=1}^\vc \lambda_j a_j \ \text{in $L^2$},\  m_j \ \text{is an $(p,q,\beta,\delta)_L$-atom and  $\sum_j |\lambda_j|^p<\infty$}\}
\end{equation}
with respect to the norm
$$
\|f\|_{H^{p,q}_{L,mol}(\RR{n})}=\inf\Big\{\Big[\sum_j |\lambda_j|^p\Big]^{1/p}: f=\sum_{j=1}^\vc \lambda_j a_j\Big\}.
$$

\end{Proposition}
\begin{proof}
	Combining the Lemmas \ref{Lem: almost orthog} and \ref{Lem: molecules in Hp}  together with the atomic characterization of $H^p_L(\mathbb{R}^n)$, we obtain this proposition.
\end{proof}

\subsection{Campanato spaces}\label{sec: BMO}

We now recall the definition of Campanato spaces associated to the Schr\"odinger operators.
\begin{Definition}
	Let $\alpha\in[0,1)$. We set
	$$ \CMS{\alpha} =\big\{ f\in L^1_{loc} : \NormA{f}_{\CMS{\alpha}}<\infty\big\}$$
	where $\NormA{f}_{\CMS{\alpha}}$ is the infimum of all $C>0$ such that
	\begin{align*}
	\f{1}{|B|^{1+\alpha/n}}\int_{B}\AbbsA{f-f_{B}} \le C
	\end{align*}
 for all balls $B$, and
	\begin{align*}
	\f{1}{|B|^{1+\alpha/n}}\int_{B} |f| \le C 
	\end{align*}
	for all balls $B$ with $r_B \geq  \rho_B$.
	
\end{Definition}
Note that in the particular case when $\alpha=0$, the Campanato space $\CMS{\alpha}$ turns out to be the BMO space $BMO_L$ which introduced in \cite{DGMTZ}. For the general case when $\alpha\in(0,1)$, these spaces were first introduced in \cite{BHS1} to consider the boundedness of generalized fractional integrals $L^{-\gamma/2}, \gamma>0$ related to  Schr\"odinger operators whose potentials satisfy certain reverse H\"older inequality. Recently, the theory of generalized Morrey-Campanato spaces associated to admissible functions has been investigated in \cite{YYZ, YYZ1}. These spaces include the Campanato type spaces in various settings of Schr\"odinger operators such as Schr\"odinger operators, degenerate Schr\"odinger operators on $\mathbb{R}^n$ and Schr\"odinger operators on Heisenberg groups and connected and simply connected nilpotent Lie groups.

It is clear from their definitions that $\CMS{\alpha}\subset\CMS{}$ and that for $\alpha=0$ we have $\CMS{\alpha}=BMO_L$.
Furthermore for $\alpha>1$, the spaces $\CMS{\alpha}$ contain only constant functions. They also coincide with the space of Lipschitz continuous functions. Indeed if we define $\Lip{\alpha}_L$ to be the space of continuous functions $f$ for which
$$ \NormA{f}_{\Lip{\alpha}_L} := \sup_{x\ne y} \frac{|f(x)-f(y)|}{|x-y|^\alpha} + \sup_{x\in\RR{n}} |\rho(x)^{-\alpha}f(x)|$$
is finite, then $\CMS{\alpha}$ and $\Lip{\alpha}_L$ coincide for all $0<\alpha\le 1$ with equivalent norms. See for example \cite{BHS, YYZ, YYZ1}.

It is important to note that the Campanato spaces are the duals of the Hardy spaces. In fact, in the case $p=1$, it was proved in \cite{DGMTZ} that $(H^1_L)^* = BMO_L$. For $p\in (\frac{n}{n+1},1)$, we have
\begin{equation}\label{DualHardyBMO}
(H^p_L)^* = \CMS{n(\frac{1}{p}-1)}.
\end{equation}
See for example \cite{YYZ}. For the predual space of the Hardy spaces  $H^1_L$ we have the following result in \cite[Theorem 4.1]{K}.

\begin{Proposition}\label{prop-CMO}
	Let $CMO_L$ be  the closure of $C^\infty(\mathbb{R}^n)$ in $BMO_L$. Then, $H^1_L$ is the dual space of $CMO_L$.
\end{Proposition}

We will summarize some properties involving the ${\rm BMO}_L^\alpha$ spaces.

\begin{Proposition}\label{prop-BMO}
	Let $\alpha\geq 0$ and $p\in [1,\vc)$. Then the following statement holds:
	\begin{enumerate}[(i)]
		\item A function $f$ belongs to the $BMO_L^\alpha$ space if and only if
		\begin{equation}\label{eq1-BMO}
		\sup_{B: {\rm ball}}\Big(\f{1}{|B|^{1+p\alpha/n}}\int_B|f(x)-f_B|^pdx\Big)^{1/p} +\sup_{B: r_B\geq \rho_B}\Big(\f{1}{|B|^{1+p\alpha/n}}\int_B|f(x)|^pdx\Big)^{1/p}<\vc.
		\end{equation}
		Moreover, the left hand side of (\ref{eq1-BMO}) is comparable with $\|f\|_{BMO_L^\alpha}$.
		
		\item For all balls $B:=B(x_0, r)$ with $r<\rho(x_0)$ and $f\in BMO_L^\alpha$, we have
		$$
		\f{1}{|B|^{1+\alpha/n}}\int_{B}|f(x)|dx\lesi \begin{cases}
		\Big(\f{\rho(x_0)}{r}\Big)^{\alpha } \|f\|_{BMO_L^\alpha}, &\alpha>0\\
		\Big[1+\log\Big(\f{\rho(x_0)}{r}\Big)\Big] \|f\|_{BMO_L^\alpha}, &\alpha=0.
		\end{cases}
		$$
		\item For all $x \in \RR{n}$ and $0 < r_1 < r_2$,
		$$
		|f_{B(x,r_1)}-f_{B(x,r_2)}|\lesi \begin{cases}
		\Big(\f{r_2}{r_1}\Big)^{\alpha }|B(x,r_1)|^{\alpha/n} \|f\|_{BMO_L^\alpha}, &\alpha>0\\
		\Big[1+\log\Big(\f{r_2}{r_1}\Big)\Big]\|f\|_{BMO_L}, &\alpha=0.
		\end{cases}
		$$
	\end{enumerate}
\end{Proposition}
\begin{proof}
	For the proof, we refer the reader to Lemma 2.2 and Lemma 2.4 in \cite{YYZ}.
\end{proof}

\section{Proof of the $T1$ criterions for $H^p_L(\mathbb{R}^n)$ and $BMO^\alpha_L(\mathbb{R}^n)$}

Before coming to the proof of the main result, we would like to give the definition of $T^*f$ for $f\in BMO_L^\alpha, 0<\alpha\leq 1$ and $T\in GCZO(\gamma,\theta)$. Let $K^*(x,y)$ be an associated kernel of $T^*$. Following the ideas in \cite{MSTZ}, we can define $T^*f$ for $f\in BMO_L^\alpha, 0<\alpha\leq 1$. For the sake of convenience, we just sketch it here.

Fix $x_0\in \mathcal{R}^n$. For  $R>\rho(x_0)$ we define
$$
T^*f(x)=T^*(f\chi_{B(x_0,R)})(x)+\int_{B(x_0, R)^c}K^*(x,y)f(y)dy.
$$
Since $f\chi_{B(x_0,R)}\in L^{\theta'}_c$ and $T^*$ is bounded on $L^{\theta'}$, the first term is well-defined.

For the second term, using \eqref{cond1}, Proposition \ref{prop-BMO} and Lemma \ref{Lem1: rho} (i) we can dominate the second term by
$$
CR^\alpha\|f\|_{BMO_L^\alpha}.
$$
Similarly to  \cite{MSTZ}, we can show that $T^*f$ is independent of $R$ in the sense that if $B(x_0,R)\subset B(x_0',R')$ then the definition using $B(x_0',R')$ coincides with the one using $B(x_0,R)$ for a.e. $B(x_0,R)$.

Since $1\in BMO_L^\alpha$, the definition above is valid for $T^*1$.

Now for $f\in BMO_L^\alpha, 0<\alpha\leq 1$. For any ball $B$ we have
$$
f=(f-f_B)\chi_{4B}+(f-f_B)\chi_{(4B)^c}+f_B:=f_1+f_2+f_3.
$$
Arguing similarly to \cite{MSTZ}, we also obtain that
$$
T^*f=T^*f_1+T^*f_2+T^*f_3.
$$

We are now ready to give the proof of Theorem \ref{Th: CZO}.

\begin{proof}[\bf Proof of Theorem \ref{Th: CZO}]

\noindent{\underline{Proof of ``if part'' for (a) and (b)}}

For $p\in (0, 1]$ with $\gamma>n(1/p-1)$, it suffices to show that $T$ maps $(p,\theta)_L$-atoms into molecules $(p,\theta,\epsilon)_L$-molecules as $p=1$ and into $(p,\theta,\epsilon,\delta)_L$-molecules as $p<1$ with $0<\epsilon<\gamma-n(1/p-1)$ and $\delta=n(1/p-1)$.

Indeed, let $a$ be an $(p,\theta)_L$-atom associated to a ball $B$.
	We first prove the size condition on $Ta$. If $j=0,1,2,3$ then $L^\theta$-boundedness of $T$ implies that
	$$ \NormA{Ta}_{L^\theta(4B)}\lesssim \NormA{a}_{L^\theta} \lesssim \AbbsA{B}^{1/\theta-1/p}. $$
For $j\geq 4$ we consider two cases:

\noindent{\emph{Case 1: $r_B<\rho_B/4$.}}

In this situation by using the cancelation property, Minkowski's inequality and \eqref{cond2} we can write
 	\begin{align*}
		\NormA{Ta}_{L^\theta(U_j(B))}
		&= \ContainC{\int_{U_j(B)}\ContainC{\int \AbbsA{K(x,y)-K(x,x_B)}\AbbsA{a(y)}\,dy}^\theta \,dx}^{1/\theta}  \\
		&\leq \int_B\Big(\int_{U_j(B)} \Big|K(x,y)-K(x,x_B)\Big|^\theta dx \Big)^{1/\theta} |a(y)|dy \\
        &\lesssim 2^{-j\gamma}|2^jB|^{-1/\theta'}\|a\|_{L^1}\\
        &\lesssim 2^{-j\gamma}|2^jB|^{-1/\theta'}|B|^{1-1/p}=2^{-j[\gamma-n(1-1/p)]}|2^jB|^{-1/p\theta'}|2^jB|^{1/\theta-1/p}.
	\end{align*}

\noindent{\emph{Case 2: $\rho_B/4\leq r_B\leq \rho_B$.}}

In this situation, by Minkowski's inequality we write
	\begin{align*}
		\NormA{Ta}_{L^\theta(U_j(B))}
		&= \ContainC{\int_{U_j(B)} \Big|\int_B K(x,y) a(y)\,dy\Big|^\theta dx}^{1/\theta} \leq\int_B\Big(\int_{U_j(B)}|K(x,y)|^\theta dx\Big)^{1/\theta}|a(y)|dy.
		\end{align*}
This along with \eqref{cond1} yields that, for $N>\gamma$,
$$
\begin{aligned}
		\|Ta\|_{L^\theta(U_j(B))}
		&\lesssim |2^jB|^{-1/\theta'}\Big(\f{\rho_B}{2^jr_B}\Big)^N \|a\|_{L^1}
    \lesssim 2^{-j[\gamma-n(1/p-1)]}|2^jB|^{1/\theta-1/p}.
\end{aligned}
$$

	To obtain the cancellation for $Ta$, we make the following split
	\begin{align*}
	\Big|\int Ta(x)\,dx|big|
	&=\Big|\int a(x) T^*1(x)\,dx\Big|\\
	&\le \int \AbbsA{a(x)}\AbbsA{T^*1(x)-(T^*1)_B}\,dx +\Big|\int a(x)\,dxBig|\AbbsA{(T^*1)_B} \\
	&=:I+II
\end{align*}

To estimate the first term, we may apply H\"older's inequality to obtain, for $p<1$,
	\begin{align*}
	I
	&\le \NormA{a}_\theta \big\Vert T^*1 - (T^*1)_B\big\Vert_{L^{\theta'}(B)}\\
	&\le |B|^{1-1/p} \ContainC{\fint_B\AbbsA{T^*1(x)-(T^*1)_B}^{\theta'}\,dx}^{1/\theta'} \\
	&\lesssim |B|^{1-1/p} \ContainC{\frac{r_B}{\rho_B}}^{n(1/p-1)}
	\end{align*}
	
	If $p=1$ then
	\begin{align*}
		I\le \ContainC{\fint_B\AbbsA{T^*1(x)-(T^*1)_B}^{\theta'}\,dx}^{1/\theta'}	 \lesssim \frac{1}{\log\big(\frac{\rho_B}{r_B}\big)}
	\end{align*}
To estimate the second term, we note that if $r_B\le \rho_B/4$ then $\int a=0$ and hence $II=0$. Otherwise we have $\rho_B/4\le \frac{r_B}{\rho_B}\le \rho_B$ and therefore
\begin{align*}
II
\le \AbbsA{(T^*1)_B}\Big|\int a(x)\,dx\Big|
\lesssim |B|^{1-1/p}
\lesssim  |B|^{1-1/p} \ContainC{\frac{r_B}{\rho_B}}^{n(1/p-1)}
\end{align*}
If $p=1$ then we argue similarly but use
$$ II\lesssim \ContainC{\frac{r_B}{\rho_B}}^\delta \lesssim \frac{1}{\log\big(\frac{\rho_B}{r_B}\big)} $$
for any $\delta>0$.

\underline{Proof of `only if' of (a)}

We borrow some ideas in \cite{MSTZ}. Assume that $T$ is bounded on $H^1_L$ then from \eqref{DualHardyBMO} $T^*$ is bounded on $BMO_L$. For $x_0\in \RR{n}$ and $0<s\leq \rho(x_0)$ we define
$$
g_{x_0,s}(x)=\chi_{[0,s]}(|x-x_0|)\log\left(\f{\rho(x_0)}{s}\right)+\chi_{(s,\rho(x_0)]}(|x-x_0|)\log\left(\f{\rho(x_0)}{|x-x_0|}\right).
$$
Then we have $g_{x_0,s}\geq 0$ and $\|g_{x_0,s}\|_{BMO_L}\leq C$. See \cite[Lemma 2.5]{MSTZ}.

We now fix $x_0\in \RR{n}$ and $0<s\leq \rho(x_0)/2$. Set $B=B(x_0,s)$ and $g_0(x)=g_{x_0,s}(x)$.

We split $f_0=(f_0-(f_0)_B)\chi_{4B}+(f_0-(f_0)_B)\chi_{(4B)^c}+(f_0)_B:= f_1+f_2+(f_0)_B$ which implies that
$$
(f_0)_BT^*1(y)=T^*f_0(y)-T^*f_1(y)-T^*f_2(y).
$$
Therefore,
$$
\begin{aligned}
(f_0)_B\log\ContainC{\frac{\rho(x_0)}{s}}&\Big(\fint_B\AbbsA{T^*1(y)-(T^*1)_B}^{\theta'}\,dy\Big)^{1/\theta'}\\
&\leq \sum_{i=0,1,2}\log\ContainC{\frac{\rho(x_0)}{s}}\Big(\fint_B\AbbsA{T^*f_i(y)-(T^*f_i)_B}^{\theta'}\,dy\Big)^{1/\theta'}\\
&:=I_0+I_1+I_2.
\end{aligned}
$$
From the $BMO_L$-boundedness of $T^*$ and Proposition \ref{prop-BMO} we obtain
$$
\begin{aligned}
I_0&\lesi \log\ContainC{\frac{\rho(x_0)}{s}}\|T^*f_0\|_{BMO_L}\lesi  \log\ContainC{\frac{\rho(x_0)}{s}}\|f_0\|_{BMO_L}\\
&\lesi \log\ContainC{\frac{\rho(x_0)}{s}}:=(f_0)_B.
\end{aligned}
$$
For the contribution of $I_1$, we have
$$
\begin{aligned}
\Big( \fint_B|T^*f_1(y)-(T^*f_1)_B|^{\theta'}dy\Big)^{1/\theta'}
&\leq 2\left(\fint_B|T^*f_1(y)|^{\theta'}dy\right)^{1/\theta'}\\
&\lesi |B|^{-1/\theta'} \|T^*f_1\|_{L^{\theta'}}.
\end{aligned}
$$
This in combination with the $L^{\theta'}$-boundedness of $T^*$ and Proposition \ref{prop-BMO} implies that
$$
\Big( \fint_B|T^*f_1(y)-(T^*f_1)_B|^{\theta'}dy \Big)^{1/\theta'}\lesi \|f_0\|_{BMO_L}.
$$
Hence, $I_1\lesi (f_0)_B$.

For the last term $I_2$, using H\"older's inequality and \eqref{cond2} we have for $y\in B$
$$
\begin{aligned}
|T^*&f_2(y)-(T^*f_2)_B|\\
&\leq \f{1}{|B|}\int_B\int_{(4B)^c}|K(z,y)-K(z,u)|\,|f_0(z)-(f_0)_B|dzdu\\
&\leq \sum_{k\geq 1}\f{1}{|B|}\int_B\left(\int_{S_k(B)}|K(z,y)-K(z,u)|^\theta dz\right)^{1/\theta} \left(\int_{S_k(B)}|f_0(z)-(f_0)_B|^{\theta'}dz\right)^{1/\theta'}du\\
&\leq \sum_{k\geq 1}2^{-k\gamma}|2^kB|^{-1/\theta'} \left(\int_{S_k(B)}|f_0(z)-(f_0)_B|^{\theta'}dz\right)^{1/\theta'}\\
&\leq \sum_{k\geq 1}2^{-k\gamma}\left[ \left(\f{1}{|2^kB|}\int_{2^kB}|f_0(z)-(f_0)_B|^{\theta'}dz\right)^{1/\theta'}+|(f_0)_{2^kB}-(f_0)_{B}|\right]
\end{aligned}
$$
which along with Proposition \ref{prop-BMO} yields that
$$
\begin{aligned}
|T^*f_2(y)-(T^*f_2)_B|&\leq \sum_{k\geq 1}2^{-k\gamma}\log(2^k)\|f_0\|_{BMO_L}\leq C.
\end{aligned}
$$
Hence, $I_2\lesi (f_0)_B$.

Taking the estimates of $I_0, I_1$ and $I_2$ into account implies that
$$
\log\ContainC{\frac{\rho(x_0)}{s}}\Big(\barint_B\AbbsA{T^*1(y)-(T^*1)_B}^{\theta'}\,dy\Big)^{1/\theta'}\leq C.
$$
This completes our proof.

\underline{Proof of `only if' of (b)}

Assume that $T$ is bounded on $H^p_L(\mathbb{R}^n)$.  We point out that in Section 3.1 of \cite{MSTZ}, they provided a definition of
$T^*f(x)$ for a.e. $x\in B(x_0,R)$, for $f\in BMO^{n(1/p-1)}_{L}$, $R\geq \rho(x_0)$ and $x_0\in\mathbb{R}^n$. Hence, by Proposition \ref{prop equiv finite atom}, suppose $g=\sum_{j=1}^n a_j \in H^p_L(\mathbb{R}^n)$, where each $a_j$ is an $(p,q)_L$-atom if $q<\infty$ and continuous $(p,q)_L$-atom if $q=\infty$. Then  we obtain that for every $f\in BMO^{n(1/p-1)}_{L}$,
\begin{align*}
\langle T^*f,g  \rangle = \langle f, Tg  \rangle &\lesssim \|f\|_{ BMO^{n(1/p-1)}_{L}} \|Tg\|_{H^p_L(\mathbb{R}^n)} \lesssim \|f\|_{ BMO^{n(1/p-1)}_{L}} \|g\|_{H^p_L(\mathbb{R}^n)}.
\end{align*}
Taking the supremum over all $g$ gives
$$   \|T^*f\|_{ BMO^{n(1/p-1)}_{L}} \lesssim \|f\|_{{ BMO^{n(1/p-1)}_{L}}}. $$
This implies that $T^*$ is bounded on $BMO^{\alpha}_{L}$ with $\alpha=n(1/p-1)$. 

For $x_0\in \RR{n}$ and $0<s<\rho(x_0)$, we define
$$
g_{x_0,s}(x)=\chi_{[0,s]}(|x-x_0|)(\rho(x_0)^\alpha-s^\alpha)+\chi_{(s,\rho(x_0)]}(|x-x_0|)(\rho(x_0)^\alpha-s^\alpha).
$$
We then have $g\geq 0$ and $\|g_{x_0,s}\|_{BMO_L^\alpha}\leq C$. See \cite[Lemma 2.5]{MSTZ}.

The remainder of the proof is similar to that of the ``only if'' part for (a). The difference here is that we must apply for the function $g_{x_0,s}$ instead of $f_0$. Hence, we omit details here.

\textbf{Proof of Theorem \ref{Th: CZO} (c)}

\smallskip
\underline{Proof of `if'}

We first recall the notion of the classical $(p,q,\epsilon)$-molecule for $H^p$ with $\f{n}{n+1}<p\leq 1$. For  $\f{n}{n+1}<p\leq 1\leq q<\vc$ and $\epsilon>0$, a function $m$ is said to be a $(p,q,\epsilon)$-molecule  if there holds
\begin{enumerate}[\upshape(i)]
	\item $\NormA{m}_{L^q(U_j(B))}\le 2^{-j\epsilon} \AbbsA{2^j B}^{1/q-1/p}$ for all $j\ge 0$
	\item $\displaystyle\int m(x)dx=0$.
\end{enumerate}
It is well-known that if $m$ is a $(p,q,\epsilon)$-molecule then $\|m\|_{H^p}\leq C$. Hence, to prove this part it suffices to prove that $T$ maps each $(p,\theta)_L$ atom into $(p,\theta,\epsilon)$ molecule for some $\epsilon>0$.

Indeed, let $a$ be a $(p,\theta)_L$ atom associated to $B$. We consider two cases:  $r_B<\rho_B/4$ and $\rho_B/4\leq r_B\leq \rho_B$. The first case is very standard. Hence, we need to consider the second case $\rho_B/4\leq r_B\leq \rho_B$. 

We first observe that from the condition $T^*1=0$ we have $\displaystyle \int Ta(x)dx=0$. To complete the proof, we need only to show that
\begin{equation}\label{eq1-thm1.2}
\|Ta\|_{L^\theta(U_j(B))}\lesi 2^{-j\epsilon}|2^jB|^{1/\theta-1/p}, \ \ j\geq 0.
\end{equation}
From the $L^\theta$-boundedness of $T$ it can be verified that \eqref{eq1-thm1.2} holds true for $j=0,1,2.$

Fix $N>n(1/p-1)$. For $j\geq 3$ by \eqref{cond1} and Minkowski's inequality we have
$$
\begin{aligned}
\|Ta\|_{L^\theta(U_j(B))}&\leq \left[\int_{U_j(B)} \left|\int_B |K(x,y)a(y)dy\right|^\theta dx\right]^{1/\theta}\\
&\leq  \int_B \left[\int_{U_j(B)}|K(x,y)|^\theta dx\right]^{1/\theta} |a(y)|dy\\
&\lesi |2^jB|^{-1/\theta'}2^{-jN}\|a\|_1\\
&\lesi |2^jB|^{-1/\theta'}2^{-jN}|B|^{1-1/p}:=2^{-j[N-n(1/p-1)]}|2^jB|^{1/\theta-1/p}.
\end{aligned}
$$
This proves \eqref{eq1-thm1.2}.

\smallskip
\underline{Proof of `only if'}
	Assume that $T$ is bounded from $H^p_L$ into $H^p$. Then by duality, $T^*$ maps $BMO^\alpha$ into $\CMS{\alpha}$ continuously with $\alpha=n(1/p-1)$. Then we have
	$$
	\|T^*1\|_{\CMS{\alpha}}\leq C\|1\|_{BMO^\alpha}=0.
	$$
	Hence, $\|T^*1\|_{\CMS{\alpha}}=0$. From the definition of $\CMS{\alpha}$ we have
	$
	\int_B|T^*1| =0
	$
	for all $B=B(x,\rho(x))$ with $ x\in \mathbb{R}^n$. This along with Lemma \ref{Lem2: rho} implies
	$
	\int_{\mathbb{R}^n}|T^*1| =0.
	$
	It follows  that $T^*1=0$.
\end{proof}

\begin{proof}[\bf Proof of Theorem \ref{thm2}]
Since $T\in GCZO^*_\rho(\gamma,\theta')$ implies  that $T^*\in GCZO_\rho(\gamma,\theta)$, the proof of the `if' directions follow from Theorem \ref{Th: CZO} and duality. 

We also observe that the proofs of the `only if' directions are essentially contained in the proofs of the `only if' directions in Theorem \ref{Th: CZO}.
\end{proof}

\section{Proofs of Applications}

In this section we give the proofs of Theorems \ref{thm3 appl1}--\ref{thm6 appl4}. 

%\subsection{Maximal operators, square functions and Laplace transform type multipliers}\label{sec: misc ops}
\subsection{Laplace transform type multipliers}\label{sec: misc ops}

Suppose $L$ is the Schr\"odinger operator defined as in \eqref{L}. 
%We now recall  the definitions for the maximal operators, the square functions and the Laplace transform type multipliers. 
%
%The \emph{maximal operators $\mathcal{W}^*$} for the heat-diffusion semigroup $e^{-tL}$ is defined by
%$$ \mathcal{W}^*(f)(x) := \sup_{t>0} |e^{-tL}(f)(x)|. $$
%
%For $0<\beta<1$, the \emph{maximal operators $\mathcal{P}^{\beta,*}$} for the generalized Poisson  semigroup  $\mathcal{P}^\beta_t$ is defined by
%$$  \mathcal{P}^{\beta,*}(f)(x) := \sup_{t>0} |\mathcal{P}^\beta_t(f)(x)|, $$
%where $\mathcal{P}^\beta_t$ are the operators defined by 
%$$  \mathcal{P}^\beta_tf(x):= {t^{2\beta}\over 4^\beta \Gamma(\beta)} \int_0^\infty  e^{- {t^2\over 4r}} e^{-rL}f(x) {dr\over r^{1+\beta}}.  $$
%When $\beta=\f{1}{2}$ we write $\mathcal{P}^{\f{1}{2}}_t = e^{-t\sqrt{L}}$.
%
%
%The \emph{Littlewood--Paley $g$-function for the heat-diffusion semigroup}  $e^{-tL}$ is defined as 
%$$  g_\mathcal{W}(f)(x):= \bigg( \int_0^\infty  |t\partial_te^{-tL}(f)(x)|^2 {dt\over t} \bigg)^{1\over2}.  $$
%
%The \emph{Littlewood--Paley $g$-function associated with the Poisson semigroup} $\mathcal{P}^{\f{1}{2}}_t = e^{-t\sqrt{L}}$ is defined as 
%$$  g_\mathcal{P}(f)(x):= \bigg( \int_0^\infty  |t\partial_te^{-t\sqrt{L}}(f)(x)|^2 {dt\over t} \bigg)^{1\over2}.  $$
Given a bounded function $a: [0,\vc) \rightarrow \mathbb{C}$, we define the \emph{Laplace transform type multipliers $m(L)$} by
\begin{equation}
\label{eq-Spectral}
m(L)f(x)=\int_0^\vc a(t)Le^{-tL}f(x)dt
\end{equation}
which is bounded on $L^2$. An example are the imaginary powers $m(L)=L^{i\nu}$ given by $a(t) = -\f{1}{\Gamma(i\nu)}t^{-i\nu}$ for $\nu\in\RR{}$.

\begin{proof}[\bf Proof of Theorem \ref{thm3 appl1}]

We now apply Theorem \ref{Th: CZO} to prove Theorem \ref{thm3 appl1}. 
%For the sake of simplicity, we only write out the details of the proof for Laplace transform type multipliers $m(L)$, and the proofs for the other four operators follows similarly from the details that we provide here and from the kernel estimates in  \cite{MSTZ}.

Denote by $m(L)(x,y)$ the associated kernel of $m(L)$. Then it was proved in \cite{MSTZ} that
\begin{Proposition}\label{Kernel-spectral}
	Let $x,y,z\in \RR{n}$ and $N>0$. Then
	\begin{enumerate}[$(a)$]
		\item $\displaystyle |m(L)(x,y)|\leq \f{C}{|x-y|^n}\left(1+\f{|x-y|}{\rho(x)}+\f{|x-y|}{\rho(y)}\right)^{-N}$;
		\item $\displaystyle |m(L)(x,y)-m(L)(x,z)|+|m(L)(y,x)-m(L)(z,x)|\leq C_\delta\f{|y-z|^\delta}{|x-y|^{n+\delta}}$, for all $|x-y|>2|y-z|$ and any $0<\delta<\sigma_0$.
	\end{enumerate} 
\end{Proposition}

	Fix $\f{n}{n+\sigma_0\wedge 1}<p\leq 1$ and take $\delta<\sigma_0\wedge 1$ so that  $\f{n}{n+\delta}<p\leq 1$. From Proposition \ref{Kernel-spectral}, $m(L)\in GCZO(\delta,2)$. Hence, in the light of Theorem \ref{Th: CZO} and the fact that $m(L)^*=\overline{m}(L)$ it suffices to prove that
	\begin{equation}\label{eq1-spectral}
	\log\ContainC{\frac{\rho_B}{r_B}}\Big(\barint_B\AbbsA{m(L)1(y)-(m(L)1)_B}^{2}\,dy\Big)^{1/2} \le C,
	\end{equation}
	\begin{equation}\label{eq2-spectral}
	\ContainC{\frac{\rho_B}{r_B}}^{n(1/p-1)}\Big(\barint_B\AbbsA{m(L)1(y)-(m(L)1)_B}^{2}\,dy\Big)^{1/2} \le C
	\end{equation}
	for every ball $B$ with $r_B\le\tfrac{1}{2}\rho_B$.
	
	Indeed, we have by Minkowski's inequality
	$$
	\begin{aligned}
	\Big(\barint_B\AbbsA{m(L)1(y)-(m(L)1)_B}^{2}\,dy\Big)^{1/2}&\lesi \Big(\barint_B\left|\barint_B m(L)1(y)-m(L)1(z)dz\right|^{2}\,dy\Big)^{1/2}\\
	&\lesi \barint_B\Big(\barint_B\left| m(L)1(y)-m(L)1(z)\right|^{2}dy\Big)^{1/2}\,dz.
	\end{aligned}
	$$
	It was proved in the proof of \cite[Proposition 4.11]{MSTZ} that
	$$
	|m(L)1(y)-m(L)1(z)|\lesi \left(\f{r_B}{\rho_B}\right)^\delta\log\left(\f{\rho_B}{r_B}\right).
	$$
	Hence, 
	$$
	\Big(\barint_B\AbbsA{m(L)1(y)-(m(L)1)_B}^{2}\,dy\Big)^{1/2}\lesi \left(\f{r_B}{\rho_B}\right)^\delta\log\left(\f{\rho_B}{r_B}\right).
	$$
	This proves \eqref{eq1-spectral} and \eqref{eq2-spectral}.
\end{proof}

\subsection{Riesz transforms $\nabla L^{-1/2}$ and $\nabla^2L^{-1}$}
Suppose $L$ is the Schr\"odinger operator defined as in \eqref{L}. 
For $i,j=1,\dots,n$, the $i$-th Riesz transform is defined by
$$
\mathcal{R}_i=\partial_{x_i}L^{-1/2}=\f{1}{\pi}\int_0^\vc \partial_{x_i}e^{-tL}\f{dt}{\sqrt{t}},
$$
and the 
$i,j$-th Riesz transform is defined by
$$
\mathcal{R}_{ij}=\partial_{x_i}\partial_{x_j}L^{-1}=\int_0^\vc \partial_{x_i}\partial_{x_j}e^{-tL}\,dt.
$$
For simplicity we shall write $\nabla$ and $\nabla^2$ for $\partial_{x_i}$ and $\partial_{x_i}\partial_{x_j}$ respectively,  and set $\RT:=\nabla L^{-1/2}$ and $\RRT:=\nabla^2 L^{-1}$.

\begin{proof}[\bf Proof of Theorem \ref{thm4 appl2}]
We first consider $\RT$.
	Now $\RT\in GCZO(\delta,2)$ for any $0<\delta<\min\{\sigma_0,1\}$. Indeed, it is well-known that $\RT$ is bounded on $L^2$. The condition \eqref{cond1} and \eqref{cond2} follow from \cite[Lemma 7]{BHS} and \cite{GLP}, respectively. On the other hand, it is obvious $\RT^*1=0$. The conclusion of the theorem follows immediately by applying Theorem \ref{Th: CZO} (c).

We now consider $\RRT$.

	We will show that $\RRT\in GCZO_\rho(\delta,\sigma)$ for any $0<\delta<\min\{\sigma_0,1\}$. Then observing that $\RRT^*1=0$, the conclusion of the theorem follows from Theorem \ref{Th: CZO} (c) also. The boundedness of $\RRT$ on $L^\sigma(\RR{n})$ for $n\ge 3$ was established in \cite{Sh}. It remains to prove \eqref{cond1} and \eqref{cond2}. The following kernel estimates are required.
	\begin{Proposition} \label{prop: RRP est}
	For each $1\le\theta\le \sigma$, there exists $\kappa>0$ such that the following holds for all $N>0$.
	\begin{enumerate}[\upshape(a)]
	\item For every $y\in \RR{n}$, $t>0$,
	\begin{align*}
		\Bigl\Vert \nabla^2p_t(\cdot,y)e^{\f{|\cdot-y|^2}{\kappa t}}\Bigr\Vert_{L^\theta}\le C t^{-1-\f{n}{2\theta'}} \Bigl(1+\f{\sqrt{t}}{\rho(y)}\Bigr)^{-N}.
\end{align*}
		\item For all $|y-y'|\le \sqrt{t}$ and any $0<\sigma_1<\sigma_0$ we have
	\begin{align*}
		\Bigl\Vert [\nabla^2p_t(\cdot,y)-\nabla^2p_t(\cdot,y')]e^{\f{|\cdot-y|^2}{\kappa t}}\Bigr\Vert_{L^\theta}\le C \Bigl(\f{|y-y'|}{\sqrt{t}}\Bigr)^{\sigma_1}t^{-1-\f{n}{2\theta'}} \Bigl(1+\f{\sqrt{t}}{\rho(y)}\Bigr)^{-N}. 
\end{align*}
	\end{enumerate}
	\end{Proposition}
	\begin{proof}[Proof of Proposition \ref{prop: RRP est}]
	Part (a) was proved in \cite{FKL2} Proposition 2.4. Part (b) can be obtained by the same argument but using the second estimate of Proposition \ref{prop-kernelestimates} in place the first. 
	\end{proof}
	We these estimates in hand, we can now obtain \eqref{cond1} and \eqref{cond2} for the kernel of $\RRT$ given by 
	$$ K(x,y) = \int_0^\infty \nabla^2_x p_t(x,y)\,dt. $$
	In fact the proof of \eqref{cond1} and \eqref{cond2} is the same as that of $K_s(x,y)$ for the operator $V^sL^{-s}$  for $s=1$ (see \eqref{Ks est}--\eqref{Ks est3} below), but applying Proposition \ref{prop: RRP est} in place of Proposition \ref{prop: VP est}. 
\end{proof}

\subsection{Riesz transforms $V^s L^{-s}$, $0<s\leq 1$, and their adjoints}
For each $0<s\le 1$ we set
\begin{align*}
	V^s L^{-s} = \f{1}{\Gamma(s)} \int_0^\infty V^s e^{-tL} \f{dt}{t^{1-s}}.
\end{align*}
It is known from Corollary 3 of \cite{Su} that the operators $V^sL^{-s}$ are bounded on $L^p(\RR{n})$ for each $1<p<\f{\sigma}{s}$.

\begin{proof}[\bf Proof of Theorem \ref{thm5 appl3}]

To prove this theorem we shall apply Theorem \ref{Th: CZO} to $T=V^sL^{-s}$. We first show that $V^sL^{-s} \in GCZO_\rho (\gamma,\theta)$ for any $1< \theta< \sigma/s$ and $0<\gamma<\sigma_0$. 
To do so, we require the following kernel estimates for $V^s e^{-tL}$.

\begin{Proposition} \label{prop: VP est}
	For each $0<s\le 1$ and $1\le\theta\le \f{\sigma}{s}$, there exists $\kappa>0$ such that the following holds for all $N>0$.
	\begin{enumerate}[\upshape(a)]
	\item For every $y\in \RR{n}$, $t>0$,
	\begin{align*}
		\Bigl\Vert V^s(\cdot)p_t(\cdot,y)e^{\f{|\cdot-y|^2}{\kappa t}}\Bigr\Vert_{L^\theta}\le C t^{-s-\f{n}{2\theta'}} \Bigl(1+\f{\sqrt{t}}{\rho(y)}\Bigr)^{-N}.
\end{align*}
		\item For all $|y-y'|\le \sqrt{t}$ and any $0<\sigma_1<\sigma_0$ we have
	\begin{align*}
		\Bigl\Vert V^s(\cdot)[p_t(\cdot,y)-p_t(\cdot,y')]e^{\f{|\cdot-y|^2}{\kappa t}}\Bigr\Vert_{L^\theta}\le C \Bigl(\f{|y-y'|}{\sqrt{t}}\Bigr)^{\sigma_1}t^{-s-\f{n}{2\theta'}} 
		%\Bigl(1+\f{\sqrt{t}}{\rho(y)}\Bigr)^{-N}. 
\end{align*}
	\end{enumerate}
\end{Proposition}
\begin{proof}[Proof of Proposition \ref{prop: VP est}]
	We need the following  estimate: for $N$ large enough we have
	\begin{align}\label{V key est}
		\Bigl(1+\f{\sqrt{t}}{\rho(x)}\Bigr)^{-N} \biggl(t\barint_{B(x,\sqrt{t})}V\biggr)^q \le C_{N,q}.
	\end{align}
	We can see this by applying Remark \ref{rem: V estimate}. 
	
	For the proof of (a), by applying the bounds on the heat kernel $p_t(x,y)$ from Proposition \ref{prop-kernelestimates} and by taking $\kappa$ large enough we have
	\begin{align*}
		\Bigl\Vert V^s(\cdot)p_t(\cdot,y)e^{\f{|\cdot-y|^2}{\kappa t}}\Bigr\Vert_{L^\theta}^\theta
		&\lesssim t^{-\f{n\theta }{2}}\Bigl(1+\f{\sqrt{t}}{\rho(y)}\Bigr)^{-N'\theta} \int V(x)^{s\theta} e^{-c\f{|x-y|^2}{t}}\,dx.
	\end{align*}
	Now since $V^s\in \RH_{\sigma/s}$ and $\theta\le \sigma/s$ then
	\begin{align*}
		\int V(x)^{s\theta} e^{-c\f{|x-y|^2}{t}}\,dx 
		&=\int_{B(y,2\sqrt{t})}  \dots\,dx+\sum_{j=1}^\infty \int_{B(y,2^{j+1}\sqrt{t})\backslash B(y,2^j\sqrt{t})} \dots dx
		 \\
		&\lesssim t^{n/2}\biggl(\barint_{B(y,\sqrt{t})}V^s\biggr)^\theta \biggl\{1+ \sum_{j=1}^\infty e^{-c4^j} 2^{j(n_{0,s}+n-n\theta)} \biggr\}\\
		&\lesssim t^{n/2-s\theta}\biggl(\barint_{B(y,\sqrt{t})}V\biggr)^{s\theta}, 
	\end{align*}
	where in the second last step we applied the doubling property of $V^s$, with $n_{0,s}$ the doubling power ov $V^s$. In the last step we applied H\"older's inequality with exponent $1/s$. 
	
	Therefore in view of \eqref{V key est} and by choosing $N'$ large enough we obtain
	\begin{align*}
		\Bigl\Vert V^s(\cdot)p_t(\cdot,y)e^{\f{|\cdot-y|^2}{\kappa t}}\Bigr\Vert_{L^\theta}
		&\lesssim t^{-s-\f{n}{2\theta'}}\Bigl(1+\f{\sqrt{t}}{\rho(y)}\Bigr)^{-N'} \biggl(t\barint_{B(y,\sqrt{t})} V(x)\,dx\biggr)^s \\
		&\lesssim t^{-s-\f{n}{2\theta'}}\Bigl(1+\f{\sqrt{t}}{\rho(y)}\Bigr)^{-N}.
	\end{align*}
	
	To prove (b) we argue as in (a), but apply the second estimate in Proposition \ref{prop-kernelestimates}.
\end{proof}

We can now show that $T=V^sL^{-s} \in GCZO_\rho (\gamma,\theta)$ for any $1< \theta< \sigma/s$ and $0<\gamma<\sigma_0$. Let $K_s(x,y)$ be the kernel of $V^sL^{-s}$. Then 
\begin{align*}
	K_s(x,y) = \f{1}{\Gamma(s)} \int_0^\infty V^s(x) p_t(x,y)\f{dt}{t^{1-s}}
\end{align*}
We first prove \eqref{cond1}. Now let $B$ be a ball with $r_B\ge 2\rho_B$ and $y\in B(x_B,\rho_B)$. Then by Proposition \ref{prop: VP est} (a), and that $\rho(y)\sim\rho_B$, 
\begin{align} \label{Ks est}
	\bigl\Vert K_s(\cdot,y)\bigr\Vert_{L^\theta(2B\backslash B) } 
	&\lesssim \int_0^\infty \bigl\Vert V^s(\cdot)p_t(\cdot,y)\bigr\Vert_{L^\theta(2B\backslash B)}\f{dt}{t^{1-s}} \\
	&\lesssim \int_0^\infty e^{-c \f{r_B^2}{t}} t^{-1-\f{n}{2\theta'}} \bigl(1+\tfrac{\sqrt{t}}{\rho_B}\bigr)^{-N} \,dt  \notag
	=:I+II
\end{align}
where 
\begin{align*}
	I = \int_0^{r_B^2} e^{-c \f{r_B^2}{t}} t^{-1-\f{n}{2\theta'}} \bigl(1+\tfrac{\sqrt{t}}{\rho_B}\bigr)^{-N} \,dt 
	\lesssim r_B^{-\f{n}{\theta'}}\Bigl(\f{\rho_B}{r_B}\Bigr)^{2N}
\end{align*}	
and 
\begin{align*}
	II= \int_{r_B^2}^\infty e^{-c \f{r_B^2}{t}} t^{-1-\f{n}{2\theta'}} \bigl(1+\tfrac{\sqrt{t}}{\rho_B}\bigr)^{-N} \,dt 
	\le \int_{r_B^2}^\infty t^{-1-\f{n}{2\theta'}} \bigl(1+\tfrac{\sqrt{t}}{\rho_B}\bigr)^{-N} \,dt 
	\lesssim 
	r_B^{-\f{n}{\theta'}}\Bigl(\f{\rho_B}{r_B}\Bigr)^{2N}
\end{align*}
for any $N>0$. 

Let us show \eqref{cond2}. Let $B$ be any ball and $y\in B$. Then for each $k\ge 1$,
\begin{align*}
	\bigl\Vert K_s(\cdot,y) - K_s(\cdot,x_B) \bigr\Vert_{L^\theta(2^{k+1}B\backslash 2^kB) }  
	&\lesssim \int_0^\infty \bigl\Vert V^s(\cdot)p_t(\cdot,y) - V^s(\cdot)p_t(\cdot,x_B) \bigr\Vert_{L^\theta(2^{k+1}B\backslash 2^kB)} \f{dt}{t^{1-s}} \\
	&= \int_0^{r_B^2} \dots +\int_{r_B^2}^\infty \dots =: I+II
\end{align*}
Now let for any $0<\gamma<\sigma_0$ we choose firstly $\sigma_1$ such that $\gamma<\sigma_1<\sigma_0$, and secondly $\epsilon = \f{1}{2}(\gamma+\f{n}{\theta'})$.  Then by the triangle inequality, Proposition \ref{prop: VP est} (a), and the fact that $y\in B$ we have 
\begin{align}\label{Ks est2}
	I 
	\lesssim \int_0^{r_B^2} e^{-c4^k \f{r_B^2}{t}} t^{-1-\f{n}{2\theta'}}\,dt 
	\lesssim 4^{-k\epsilon} r_B^{-2\epsilon} \int_0^{r_B^2} t^{-1-\f{n}{2\theta'}+\epsilon}dt
	\lesssim 4^{-k\epsilon} r_B^{-\f{n}{\theta'}}  
	\end{align}
	We also have by Proposition \ref{prop: VP est} (b)
	\begin{align*}
	II \lesssim 	\int_{r_B^2}^\infty e^{-c4^k \f{r_B^2}{t}} t^{-1-\f{n}{2\theta'}} \Bigl(\f{|y-x_B|}{\sqrt{t}}\Bigr)^{\sigma_1}\,dt 
	\lesssim r_B^{\sigma_1-2\epsilon} \int_{r_B^2}^\infty t^{-1-\f{n}{2\theta'}-\f{\sigma_1}{2}+\epsilon}\,dt 
	\lesssim 4^{-k\epsilon}r_B^{-\f{n}{\theta'}}
	\end{align*}
Thus collecting our estimates for $I$ and $II$ we have
\begin{align}\label{Ks est3}
	\bigl\Vert K_s(\cdot,y) - K_s(\cdot,x_B) \bigr\Vert_{L^\theta(2^{k+1}B\backslash 2^kB) }
	\lesssim 4^{-k\epsilon}r_B^{-\f{n}{\theta'}}
	= 2^{-k\gamma} |2^k B|^{-\f{1}{\theta'}}
	\end{align}
	where $\gamma = 2\epsilon -\f{n}{\theta'}$. 
	
	Next we show conditions (a) and (b) of Theorem \ref{Th: CZO} for $T^*=L^{-s}V^s$. More precisely we prove
	\begin{align}\label{LV1}
		&\log\Bigl(\frac{\rho_B}{r_B}\Bigr)\Big(\barint_B\AbbsA{L^{-s}V^s1(y)-(L^{-s}V^s1)_B}^{\theta'}\,dy\Big)^{1/\theta'} \le C \\
		\label{LV2}
		&\ContainC{\frac{\rho_B}{r_B}}^{n(1/p-1)}\Big(\barint_B\AbbsA{L^{-s}V^s1(y)-(L^{-s}V^s1)_B}^{\theta'}\,dy\Big)^{1/\theta'} \le C
	\end{align}
	for every ball $B$ with $r_B\le\tfrac{1}{2}\rho_B$ and $\f{n}{n+s\sigma_0\wedge 1}<p<1$.
	In fact, for any $1<\theta<\infty$, estimates  \eqref{LV1} and \eqref{LV2} are consequences of the following stronger estimate
	\begin{align}\label{LV lip}
		\bigl| L^{-s}V^s1(x)-L^{-s}V^s1(y)\bigr| \lesssim \Bigl(\f{r_B}{\rho_B}\Bigr)^\delta 
	\end{align}
	for any ball $B$ with $r_B\le \f{1}{2}\rho_B$, and $x,y\in B$, and any $0<\delta<s\sigma_0 \wedge 1$. We shall show \eqref{LV lip} by applying the following lemma. 
	\begin{Lemma}\label{lem: EV}
	Let $0<s\le 1$. For any $0<\sigma_1<\sigma_0$, $0<\delta\le s\sigma_0\wedge 1$ and $N>0$ the following holds:
\begin{align}\label{lem: EV1}
\bigl|e^{-tL}V^s(x)\bigr| \le C t^{-s} \Bigl(\f{\sqrt{t}}{\rho(x)}\Bigr)^\delta \Bigl(1+\f{\sqrt{t}}{\rho(x)}\Bigr)^{-N} 
\end{align}
for any $x\in\RR{n}$ and $t>0$, and
\begin{align}\label{lem: EV2}
\bigl|e^{-tL}V^s(x)-e^{-tL}V^s(y)\bigr| \le C t^{-s} \Bigl(\f{\sqrt{t}}{\rho(x)}\Bigr)^\delta \Bigl(1+\f{\sqrt{t}}{\rho(x)}\Bigr)^{-N} \Bigl(\f{|x-y|}{\sqrt{t}}\Bigr)^{\sigma_1}
\end{align}
for all $t>0$ and $|x-y|\le\sqrt{t}$. 
	\end{Lemma}
	\begin{proof}
	We first prove \eqref{lem: EV1}. We have 
	\begin{align*}
		e^{-tL}V^s(x) 
		&= \int p_t(x,w)V^s(w)\,dw \\
		&\lesssim \Bigl(1+\f{\sqrt{t}}{\rho(x)}\Bigr)^{-M} \int t^{-\f{n}{2}} e^{-c\f{|x-w|^2}{t}} V^s(w)\,dw \\
		&\lesssim t^{-s}\Bigl(1+\f{\sqrt{t}}{\rho(x)}\Bigr)^{-M} \biggl(t\barint_{B(x,\sqrt{t})}V\biggr)^s. 
	\end{align*}
	Now from Remark \ref{rem: V estimate} 	and by choosing $M$ large enough, we obtain the required estimate.

	The proof of \eqref{lem: EV2} begins with
	\begin{align*}
		\bigl|e^{-tL}V^s(x)-e^{-tL}V^s(y)\bigr| 
		&\le \int \bigl|p_t(x,w)-p_t(y,w)\bigr| V^s(w)\,dw \\
		&\lesssim \Bigl(\f{|x-y|}{\sqrt{t}}\Bigr)^{\sigma_1}\Bigl(1+\f{\sqrt{t}}{\rho(x)}\Bigr)^{-M} \int t^{-\f{n}{2}} e^{-c\f{|x-w|^2}{t}} V^s(w)\,dw, 
	\end{align*}
	and we proceed as in \eqref{lem: EV1}. 
	\end{proof}
	Let us continue with the proof of \eqref{LV lip}. We first write
	\begin{align*}
		\bigl| L^{-s}V^s1(x)-L^{-s}V^s1(y)\bigr| 
		&\le \int_0^\infty \bigl| e^{-tL}V^s(x) -e^{-tL}V^s(y)\bigr| \f{dt}{t^{1-s}} \\
		&= \int_0^{4r_B^2} + \int_{4r_B^2}^{\rho_B^2} + \int_{\rho_B^2}^\infty \dots =: I+ II+III.
	\end{align*}
	Now by \eqref{lem: EV1}, and that $\rho(x)\sim\rho(y)\sim\rho_B$ we have for any $0<\delta<s\sigma_0\wedge 1$, 
	\begin{align*}
		I\le \int_0^{4r_B^2}\bigl|e^{-tL}V^s(x)\bigr|\f{dt}{t^{1-s}} + \int_0^{4r_B^2}\bigl|e^{-tL}V^s(y)\bigr|\f{dt}{t^{1-s}} 
		\lesssim \rho_B^{-\delta} \int_0^{4r_B^2} \f{dt}{t^{1-\f{\delta}{2}}} 
		\lesssim \Bigl(\f{r_B}{\rho_B}\Bigr)^\delta.
	\end{align*}
	Now pick $\delta<\delta_1\le s\sigma_0\wedge 1$. From $|x-y|\le 2r_B\le \sqrt{t}$ and that $\rho(x)\sim\rho_B$ we have by \eqref{lem: EV2},
	\begin{align*}
	 II \le \int_{4r_B^2}^{\rho_B^2} t^{-s} \Bigl(\f{|x-y|}{\sqrt{t}}\Bigr)^{\delta_1} \Bigl(\f{\sqrt{t}}{\rho_B}\Bigr)^{\delta_1}\f{dt}{t^{1-s}} 
	\lesssim \Bigl(\f{r_B}{\rho_B}\Bigr)^{\delta_1} \int_{4r_B^2}^{\rho_B^2} \f{dt}{t} 
	\lesssim \Bigl(\f{r_B}{\rho_B}\Bigr)^{\delta_1} \log \Bigl(\f{\rho_B}{r_B}\Bigr) 
	\lesssim \Bigl(\f{r_B}{\rho_B}\Bigr)^\delta.
	\end{align*}
	Finally from \eqref{lem: EV2} and by taking $N$ large enough,
	\begin{align*}
	 III\lesssim \int_{\rho_B^2}^\infty t^{-s}\Bigl(\f{|x-y|}{\sqrt{t}}\Bigr)^\delta \f{dt}{t^{1-s}} \lesssim r_B^\delta \int_{\rho_B^2}^\infty \f{dt}{t^{1+\f{\delta}{2}}} \lesssim \Bigl(\f{r_B}{\rho_B}\Bigr)^\delta.
	\end{align*}
	The terms $I, II$ and $III$ together give \eqref{LV lip}.
	
	Thus \eqref{LV1} and \eqref{LV2} hold for any $\theta>1$, and so we may conclude the proof of Theorem \ref{thm5 appl3} by invoking Theorem \ref{Th: CZO}.
\end{proof}

\subsubsection{The Riesz transforms $L^{-s}V^s$}

Before giving the proof of Theorem \ref{thm6 appl4} we make some preliminary remarks.

Firstly, the hypothesis $V\in \RH_\infty$ ensures that $V^sL^{-s}$ and $L^{-s}V^s$ are both $L^p$ bounded for all $1<p<\infty$. 
Secondly, the conditions $V\in\RH_\infty$ and \eqref{V cond1} imply
\begin{align}\label{V cond2}
	V(x) \le C' \rho(x)^{-2} \qquad \text{a.e. }x
\end{align}
for some $C'>0$. See \cite{Sh2} Remark 1.8.

Our conditions on $V$ guarantee it admits a certain smoothness, encapsulated in the following result.
\begin{Lemma} \label{lem: V lip}
	If $V$ satisfies \eqref{V cond1} then for each $0<s\le 1$ there exists $C>0$ depending only on $s$ and $V$ such that 	for every $0<\eta\le 1$ we have
	\begin{align*}
		\bigl| V^s(x)-V^s(y)\bigr| 
		\le  \f{C}{t^s}\Bigl(\f{|x-y|}{\sqrt{t}}\Bigr)^\eta \Bigl(\f{\sqrt{t}}{\rho(x)}\Bigr)^{1+2s}\Bigl(1+\f{\sqrt{t}}{\rho(x)}\Bigr)^{2+4s}
	\end{align*}
 whenever $|x-y|\le\sqrt{t}$. 
\end{Lemma}
\begin{proof}
	From the mean value theorem and part (i) of Lemma \ref{Lem1: rho}   we have, for some $x' \in B(x,|x-y|)$,  
	\begin{align*}
		\bigl|V^s(x)-V^s(y)\bigr| 
&\lesssim V^{s-1}(x') \bigl|\nabla V(x')\bigr| |x-y| \\
	&	\lesssim \rho(x')^{-1-2s}|x-y|  \\
		&\lesssim \rho(x)^{-1-2s} \Bigl(1+\f{|x-y|}{\rho(x)}\Bigr)^{2+4s} |x-y|
	\end{align*}
This yields the required result if  $|x-y|\le \sqrt{t}$.
\end{proof}

This smoothness grant us the following analogues of Proposition \ref{prop: VP est} and Lemma \ref{lem: EV} respectively.
\begin{Proposition}\label{prop: VP est2}
 Assume that $V$ satisfies \eqref{V cond1}. Then for each $0<s\le 1$, there exists $C>0$ such that the following holds for all $N>0$,
\begin{enumerate}[\upshape(a)]
	\item For every $x,y\in \RR{n}$, $t>0$,
	\begin{align*}
		\bigl| V^s(y) \,p_t(x,y)\bigr| 
		\le C t^{-s-\f{n}{2}} e^{-\f{|x-y|^2}{ct}}\Bigl(1+\f{\sqrt{t}}{\rho(y)}\Bigr)^{-N}.
\end{align*}
		\item For all $|y-y'|\le \sqrt{t}$ and any $0<\eta\le 1$ we have
	\begin{align*}
\bigl| V^s(y) \,p_t(x,y) - V^s(y') \,p_t(x,y')\bigr| 
\le C t^{-s-\f{n}{2}} e^{-\f{|x-y|^2}{ct}} \Bigl(\f{|y-y'|}{\sqrt{t}}\Bigr)^\eta
\end{align*}
	\end{enumerate}
\end{Proposition}
\begin{proof}[Proof of Proposition \ref{prop: VP est2}]
	To prove (a) we observe that from the heat kernel bounds in Proposition \ref{prop-kernelestimates} and from \eqref{V cond2} that 
	\begin{align*}
		\bigl| V^s(y) \,p_t(x,y)\bigr| 
		&\lesssim t^{-s-\f{n}{2}} e^{-\f{|x-y|^2}{ct}}\Bigl(1+\f{\sqrt{t}}{\rho(y)}\Bigr)^{-N'} \Bigl(\f{\sqrt{t}}{\rho(y)}\Bigr)^{2s}
	\end{align*}
	The result now follows by taking $N'$ large enough. 
	
	For part (b) we write 
	\begin{align*}
		\bigl| V^s(y) \,p_t(x,y) - V^s(y') \,p_t(x,y')\bigr| 
		&\le  V^s(y) \bigl| p_t(x,y) - \,p_t(x,y')\bigr| 
		+\bigl| V^s(y) - V^s(y')\bigr| \bigl| \,p_t(x,y')\bigr| \\
		&=:I+II
	\end{align*}
	From the second estimate in Proposition \ref{prop-kernelestimates} and \eqref{V cond2} we have
	\begin{align*}
		I 
		\lesssim  \rho(y)^{-2s}  t^{-\f{n}{2}} e^{-\f{|x-y|^2}{ct}} \Bigl(\f{|y-y'|}{\sqrt{t}}\Bigr)^\eta \Bigl(1+\f{\sqrt{t}}{\rho(y)}\Bigr)^{-N'} 
		\lesssim t^{-s-\f{n}{2}} e^{-\f{|x-y|^2}{ct}} \Bigl(\f{|y-y'|}{\sqrt{t}}\Bigr)^\eta
	\end{align*}
	by taking $N'$ large enough. Next we have from the bounds on the heat kernel, that $|y-y'|\le \sqrt{t}$, and Lemma \ref{lem: V lip}, 
	\begin{align*}
		II 
		&\lesssim \bigl| V^s(y) - V^s(y')\bigr| \, t^{-\f{n}{2}} e^{-\f{|x-y|^2}{ct}}  \Bigl(1+\f{\sqrt{t}}{\rho(y')}\Bigr)^{-N'} \\
		&\lesssim   t^{-s-\f{n}{2}} e^{-\f{|x-y|^2}{ct}}\Bigl(\f{|y-y'|}{\sqrt{t}}\Bigr)^\eta \Bigl(\f{\sqrt{t}}{\rho(y')}\Bigr)^{1+2s}\Bigl(1+\f{\sqrt{t}}{\rho(y')}\Bigr)^{2+4s}
		 \Bigl(1+\f{\sqrt{t}}{\rho(y')}\Bigr)^{-N'} 
	\end{align*}
	which gives the required estimate after taking $N'$ large enough. 
\end{proof}
\begin{Lemma}\label{lem: VE}
	Suppose that $V$ satisfies \eqref{V cond1} and $0<s\le 1$. Then for any $0<\delta\le 2s\wedge 1$ and $0<\eta\le 1$ the following holds:
	\begin{align}\label{lem: VE1}
\bigl|V^s(x)e^{-tL}1(x)\bigr| \le C t^{-s} \Bigl(\f{\sqrt{t}}{\rho(x)}\Bigr)^\delta \Bigl(1+\f{\sqrt{t}}{\rho(x)}\Bigr)^{-N} 
\end{align}
for any $x\in\RR{n}$ and $t>0$, and
\begin{align}\label{lem: VE2}
\bigl|V^s(x)e^{-tL}1(x)-V^s(y)e^{-tL}1(y)\bigr| \le C t^{-s} \Bigl(\f{\sqrt{t}}{\rho(x)}\Bigr)^\delta \Bigl(1+\f{\sqrt{t}}{\rho(x)}\Bigr)^{-N} \Bigl(\f{|x-y|}{\sqrt{t}}\Bigr)^\eta
\end{align}
for all $t>0$ and $|x-y|\le\sqrt{t}$. 
	\end{Lemma}
	\begin{proof}[Proof of Lemma \ref{lem: VE}]
		Firstly by \eqref{V cond2} and the bounds on the heat kernel, 
		\begin{align*}
			\bigl|V^s(x)e^{-tL}1(x)\bigr|  
			\lesssim \rho(x)^{-2s} \Bigl(1+\f{\sqrt{t}}{\rho(x)}\Bigr)^{-N'} 
		\end{align*}
		Thus \eqref{lem: VE1} follows by considering the cases $\sqrt{t}\ge \rho(x)$ and $\sqrt{t}<\rho(x)$ and taking suitable $N'$. Turning to \eqref{lem: VE2} we write
		\begin{align*}
\bigl|V^s(x)e^{-tL}1(x)-V^s(y)e^{-tL}1(y)\bigr| 
\le  \bigl|V^s(x)-V^s(y)\bigr| \bigl| e^{-tL}1(x)\bigr|
+ V^s(y)\bigl|e^{-tL}1(x)-e^{-tL}1(y)\bigr|
\end{align*}
Now from Lemma \ref{lem: V lip} we have
\begin{align*}
	\bigl|V^s(x)-V^s(y)\bigr| \bigl| e^{-tL}1(x)\bigr| 
	\lesssim  t^{-s}\Bigl(\f{|x-y|}{\sqrt{t}}\Bigr)^\eta \Bigl(\f{\sqrt{t}}{\rho(x)}\Bigr)^{1+2s}\Bigl(1+\f{\sqrt{t}}{\rho(x)}\Bigr)^{2+4s}\Bigl(1+\f{\sqrt{t}}{\rho(x)}\Bigr)^{-N'} 
\end{align*}
which gives the right hand side of \eqref{lem: VE2}. Next from \eqref{hk holder}, \eqref{V cond2}, and Lemma \ref{Lem1: rho} (i), 
\begin{align*}
	 V^s(y)\bigl|e^{-tL}1(x)-e^{-tL}1(y)\bigr|
	&\le V^s(y) \int \bigl|p_t(x,w)-p_t(y,w)\bigr| \,dw \\
	&\lesssim \rho(y)^{-2s} \Bigl(\f{|x-y|}{\sqrt{t}}\Bigr)^\eta\Bigl(1+\f{\sqrt{t}}{\rho(x)}\Bigr)^{-N'} \\
	&\lesssim \rho(x)^{-2s} \Bigl(1+\f{\sqrt{t}}{\rho(x)}\Bigr)^{4s}\Bigl(\f{|x-y|}{\sqrt{t}}\Bigr)^\eta\Bigl(1+\f{\sqrt{t}}{\rho(x)}\Bigr)^{-N'} 
\end{align*}
which also yields the right hand side of \eqref{lem: VE2}.
	\end{proof}

\begin{proof}[\bf Proof of Theorem \ref{thm6 appl4}]

We shall show that $T=L^{-s}V^s \in GCZO_\rho(\gamma,\theta)$ for any $1<\theta<\infty$ and $0<\gamma<1$. Note firstly that $V\in\RH_\infty$ implies that $L^{-s}V^s$ is bounded on $L^\theta$ for any $1<\theta<\infty$. Next we set
$$ K^*_s(x,y) = \f{1}{\Gamma(s)} \int_0^\infty p_t(x,y) \,V^s(y) \f{dt}{t^{1-s}} $$
to be the kernel of $L^{-s}V^s$. 

Let us show \eqref{cond1}. Fix a ball $B$ with $r_B\ge 2\rho_B$ and $y\in B$. Then we have $\rho(y)\sim \rho_B$. Thus from Proposition \ref{prop: VP est2}, for any $1<\theta<\infty$, 
\begin{align*}
	\bigl\Vert K^*_s(\cdot,y)\bigr\Vert_{L^\theta(2B\backslash B) } 
	\lesssim \int_0^\infty \bigl\Vert p_t(\cdot,y) V^s(y)\bigr\Vert_{L^\theta(2B\backslash B)}\f{dt}{t^{1-s}} 
	\lesssim \int_0^\infty e^{-c \f{r_B^2}{t}} t^{-1-\f{n}{2\theta'}} \bigl(1+\tfrac{\sqrt{t}}{\rho_B}\bigr)^{-N} \,dt 
\end{align*}
At this point we can continue as in \eqref{Ks est}. 

Let us turn to \eqref{cond2}. Now for each $1<\theta<\infty$ and $0<\gamma<1$ let us take $\epsilon = \f{1}{2}(\gamma+\f{n}{\theta'})$. Let $B$ be any ball and $y\in B$. Then  for each $k\ge 1$, , 
\begin{align*}
	\bigl\Vert K^*_s(\cdot,y) - K^*_s(\cdot,x_B) \bigr\Vert_{L^\theta(2^{k+1}B\backslash 2^kB) }  
	&\lesssim \int_0^\infty \bigl\Vert p_t(\cdot,y) V^s(y) - p_t(\cdot,x_B)V^s(x_B) \bigr\Vert_{L^\theta(2^{k+1}B\backslash 2^kB)} \f{dt}{t^{1-s}} \\
	&= \int_0^{r_B^2} \dots +\int_{r_B^2}^\infty \dots =: I+II
\end{align*}
We can apply Proposition \ref{prop: VP est2} (a) and proceed as in \eqref{Ks est2} to obtain
$$ I\lesssim 4^{-k\epsilon}r_B^{-\f{n}{\theta'}}$$
For the second term, Proposition \ref{prop: VP est2} (b) gives
\begin{align*}
	II \lesssim 	\int_{r_B^2}^\infty e^{-c4^k \f{r_B^2}{t}} t^{-1-\f{n}{2\theta'}} \Bigl(\f{|y-x_B|}{\sqrt{t}}\Bigr)\,dt 
	%\lesssim r_B^{\sigma_1-2\epsilon} \int_{r_B^2}^\infty t^{-1-\f{n}{2\theta'}-\f{\sigma_1}{2}+\epsilon}\,dt 
	\lesssim 4^{-k\epsilon}r_B^{-\f{n}{\theta'}}
	\end{align*}
Combining our estimates for $I$ and $II$ gives \eqref{cond2} because $\gamma=2\epsilon -\f{n}{\theta'}$. 

Next we prove that $T^*=V^sL^{-s}$ satisfies (a) and (b) of Theorem \ref{Th: CZO}. As before this follows from the following version of \eqref{LV lip}: for each ball $B$ with $r_B\le \f{1}{2}\rho_B$, 
\begin{align} \label{VL lip}
	\bigl|V^sL^{-s} 1(x)- V^sL^{-s}1(y)\bigr| \lesssim \Bigl(\f{r_B}{\rho_B}\Bigr)^\delta
\end{align}
for any $x,y\in B$ and $0<\delta<2s\wedge 1$. We can obtain \eqref{VL lip} by arguing as in \eqref{LV lip}, but using Lemma \ref{lem: VE} in place of Lemma \ref{lem: EV}. 
\end{proof}

\bigskip

\noindent
{\bf Acknowledgments.} \  T.A. Bui  is supported by  ARC DP 140100649.

\end{document}